\RequirePackage{fix-cm}
\documentclass[smallextended]{svjour3}       
\smartqed  
\usepackage{graphicx}
\usepackage{latexsym,amssymb,amsmath,enumerate,bbm,graphicx,psfrag,color}
\usepackage{stmaryrd} 
\usepackage[latin1]{inputenc}
\usepackage{cite}

\def\norm#1{\hspace{0.2ex} \|#1\| \hspace{0.2ex}}

\newcommand{\1}{\ensuremath{\mathbbm{1}}} 
\newcommand{\R}{\ensuremath{\mathbbm{R}}} 
\newcommand{\N}{\ensuremath{\mathbbm{N}}} 
\newcommand{\range}{\mathcal R}

\newcommand{\dx}[1][x]{\ensuremath{\,{\rm{d}} #1}}

\newcommand{\LL}{\mathcal L}

\newcommand{\kommentar}[1]{}

 \def\makeheadbox{\relax}
\begin{document}

\title{Uniqueness, stability and global convergence for a discrete inverse elliptic Robin transmission problem}


\titlerunning{Uniqueness, stability and global convergence for an inverse transmission problem}        

\author{Bastian Harrach}


\institute{B. Harrach \at Institute for Mathematics, Goethe-University Frankfurt, Germany\\
              \email{harrach@math.uni-frankfurt.de}          
}

\date{}

\maketitle

\begin{abstract}We derive a simple criterion that ensures uniqueness, Lipschitz stability and global convergence of Newton's method for the finite dimensional zero-finding problem of a continuously differentiable, pointwise convex and monotonic function. Our criterion merely requires to evaluate the directional derivative of the forward function at finitely many evaluation points and for finitely many directions. 

We then demonstrate that this result can be used to prove uniqueness, stability and global convergence for an inverse coefficient problem with finitely many
measurements. We consider the problem of determining an unknown inverse Robin transmission coefficient in an elliptic PDE. Using a relation to monotonicity and localized potentials techniques, we show that a piecewise-constant coefficient on an a-priori known partition with a-priori known bounds is uniquely determined by finitely many boundary measurements and that it can be uniquely and stably reconstructed by a globally convergent Newton iteration. We derive a constructive method to identify these boundary measurements, calculate the stability constant and give a numerical example.
\keywords{Uniqueness \and Lipschitz stability \and global convergence \and Robin transmission problem \and Newton method \and monotonicity method \and localized potentials}
\subclass{35R30 \and 65M32 \and 58C15}
\end{abstract}

\section{Introduction}
\label{Sec:intro}

New technologies for medical imaging, non-destructive testing, or geophysical exploration are often based on mathematical inverse coefficient problems,
where the coefficient of a partial differential equation is to be reconstructed from (partial) knowledge of its solutions.
A prominent example is the emerging technique of electrical impedance tomography (EIT), cf.\ \cite{henderson1978impedance,barber1984applied,wexler1985impedance,newell1988electric,metherall1996three,cheney1999electrical,borcea2002electrical,borcea2003addendum,lionheart2003eit,holder2004electrical,bayford2006bioimpedance,uhlmann2009electrical,martinsen2011bioimpedance,seo2013electrical,adler2015electrical}, and the references therein for a broad overview.
Inverse coefficient problems are usually highly non-linear and ill-posed, and uniqueness and stability questions, as well as the design and the theoretical study of reconstruction algorithms are extremely challenging topics in theoretical and applied research.

In practical applications, only finitely many measurements can be made, the unknown coefficient has to be parametrized with finitely many variables (e.g., by assuming piecewise-constantness on a given partition), and physical considerations typically limit the unknown coefficient to fall between certain a-priori known bounds.
Thus, after shifting and rescaling, a practical inverse coefficient problem can be put in the form of finding the zero $x\in [0,1]^n$ of a non-linear function $F:\ \R^n\to \R^m$,
\[
F(x)=0.
\]
It is of utmost importance to determine what measurements make this finite-dimen\-sion\-al inverse (or zero-finding) problem uniquely solvable, to evaluate the stability of the solution with respect to model and measurement errors, and to design convergent numerical reconstruction algorithms.

In this paper, we will study this problem under the assumption that $F$ is a pointwise monotonic and convex function, which often arises in elliptic inverse coefficient problems (cf.\ our remarks on the end of this introduction). We will derive a simple criterion that implies the existence of a zero, and the injectivity of $F$ in a certain neighborhood of $[0,1]^n$. It also allows us to quantify the Lipschitz stability constant of the left inverse $F^{-1}$ and, for $n=m$, ensures global convergence of Newton's method. The criterion is easy-to-check as it merely requires to evaluate the directional derivative $F'(z)d$ for a finite number of evaluation points $z$ and finitely many directions $d$.

We then show how our result can be applied to the inverse problem of determining a Robin transmission coefficient in an elliptic PDE from the associated Neumann-to-Dirichlet operator that is motivated by EIT-based corrosion detection \cite{harrach2019global}. We assume that the Robin coefficient is piecewise-constant on a-priori known partition of the interface into $n$ parts, and that we a-priori know upper and lower bounds for the Robin coefficient's values. We then show how to construct $n$ boundary measurements that uniquely determine the $n$ unknown Robin values, and for which a standard Newton method globally converges. We also quantify the stability of the solution with respect to errors, and numerically demonstrate our result on a simple setting.

Let us give some references to put our result in the context of existing literature. The arguably most famous elliptic inverse coefficient problem is the Calder\'on problem \cite{calderon1980inverse,calderon2006inverse}, that arises in EIT, cf.\ \cite{kohn1984determining,kohn1985determining,druskin1998uniqueness,sylvester1987global,nachman1996global,astala2006calderon,kenig2007calderon,haberman2013uniqueness,caro2016global,krupchyk2016calderon} for an incomplete list of seminal breakthroughs for the uniqueness question in an infinite-dimensional setting. 

Several recent works have addressed uniqueness and Lipschitz stability questions for the problem of determining finitely many parameters (e.g., by assuming piecewise-constantness) from finitely or infinitively many measurements in inverse coefficient problems, cf., \cite{kazemi1993stability,alessandrini1996determining,imanuvilov1998lipschitz,imanuvilov2001global,cheng2003lipschitz,alessandrini2005lipschitz,bacchelli2006lipschitz,bellassoued2006lipschitz,klibanov2006lipschitz,klibanov2006lipschitz_nonstandard,bellassoued2007lipschitz,sincich2007lipschitz,yuan2007lipschitz,yuan2009lipschitz,beretta2011lipschitz,beretta2013lipschitz,melendez2013lipschitz,alessandrini2017lipschitz,beretta2017uniqueness,beilina2017lipschitz,alessandrini2018lipschitz,alberti2019calderon,ruland2019lipschitz,harrach2019uniqueness,harrach2019global,alberti2019infinite,harrach2020monotonicity}. To the knowledge of the author, the results presented herein, is the first on explicitly calculating those measurements that uniquely determine the unknown parameters, and, together with \cite{harrach2019global}, it is the first result to explicitly calculate the Lipschitz constant for a given setting. Moreover, we obtain the unique existence of a solution also for noisy measurements, so that Lipschitz stability also yields an error estimate in the case of noise.

Reconstruction algorithms for inverse coefficient problems typically rely on Newton-type approaches 
or on globally minimizing a non-convex regularized data-fitting functional. Both approaches require an initial guess close to the unknown solution, so that most algorithms are only known to converge locally.
For the sake of rigor, it should be noted at this point, that the difficulty in this context is not to construct globally convergent methods 
but \emph{computationally feasible} globally convergent methods. To elaborate on this point, let us consider a minimization problem with a continuous functional $f:\ I\to \R$ over a finite-dimensional interval $I\subseteq \R^n$. A trivial optimization algorithm is to choose a countable dense subset $(q_m)_{m\in \N}\subset I$, and setting $x_0:=q_0$,
\[
x_m:=\left\{ \begin{array}{l l} q_m & \text{ if $f(q_m)<f(x_{m-1})$}\\
x_{m-1} & \text{ else.}
 \end{array}\right.
\]
Then, obviously, any accumulation point of $(x_m)_{m\in \N}$ is a global minimizer of $f$.
But this type of approach requires a completely unfeasible amount of function evaluations and is thus usually disregarded in practice. Note, however, that together with estimates on the convergence range of an iterative algorithm
as in the recent preprint of Alberti and Santacesaria \cite{alberti2019infinite}, and the progress of parallel computing power, these type of approaches may become feasible at least for lower dimension numbers. 

To obtain (computationally feasible) globally convergent algorithms, quasi-reversibility and convexification ideas have been developed in the 
the seminal work of Klibanov et al., cf., e.g.\ \cite{beilina2008globally,beilina2012approximate,klibanov2017convexification,klibanov2019convexification}.
Alternatively, one can use very specific properties of the considered problem, cf., e.g., the global convergence result of Knudsen, Lassas, Mueller and Siltanen \cite{knudsen2009regularized} for the d-bar method for EIT, 
or resort to only reconstructing the shape of an anomaly, cf.\ \cite{ikehata1999draw,ikehata2000reconstruction,ide2007probing,harrach2010exact,harrach2013monotonicity, harrach2016enhancing,harrach2018monotonicity} for some globally convergent approaches.
The theory developed in this work shows that, somewhat surprisingly, global convergence holds for the standard zero-finding Newton method, 
when the right measurements are being used, and this also implies fast quadratic convergence. On the other hand, so far, our theory does not allow to utilize more measurements than unknowns or to explicitly add regularization, which would be advantageous in practical applications. 
Moreover, the calculated measurements tend to become more or more oscillatory for higher dimensional problems. 
Hence, so far, we can only expect our approach to be practically feasible for relatively few unknowns where discretization sufficiently regularizes the ill-posed problem.

On the methodological side, this work builds upon \cite{harrach2019global,harrach2019uniqueness} and stems from the theory of combining monotonicity estimates with localized potentials, cf.\ \cite{harrach2009uniqueness,harrach2010exact,harrach2012simultaneous,arnold2013unique,harrach2013monotonicity,barth2017detecting,harrach2017local,brander2018monotonicity,griesmaier2018monotonicity,harrach2018localizing,seo2019learning,harrach2019fractional_I,harrach2019helmholtz,harrach2019dimension,eberle2020lipschitz,harrach2020monotonicity}
for related works, and \cite{tamburrino2002new,harrach2015combining,harrach2015resolution,harrach2016enhancing,maffucci2016novel,tamburrino2016monotonicity,garde2017convergence,su2017monotonicity,ventre2017design,garde2018comparison,harrach2018monotonicity,zhou2018monotonicity,candiani2019monotonicity,garde2019regularized,garde2020reconstruction}
for practical monotonicity-based reconstruction methods.
In this work, the monotonicity and convexity of the forward function is based on the so-called monotonicity relation
which goes back to Ikehata, Kang, Seo, and Sheen \cite{ikehata1998size,kang1997inverse}. The
existence of measurements that fulfill the extra criterion 
on the directional derivative evaluations follows from localized potentials arguments \cite{gebauer2008localized}. Hence, it might be possible to extend the theory developed herein to other elliptic inverse coefficient problems where monotonicity and localized potentials results are also available. Note however, that this extension is not obvious
since the localized potentials results for the herein considered Robin transmission problem are stronger than 
those known for other coefficient problems such as EIT.

Finally, it should be noted, that the monotonicity and localized potentials techniques evolved from the factorization method \cite{kirsch1998characterization,bruhl2000numerical,gebauer2007factorization,lechleiter2008factorization,harrach2013recent}, and that
global convergence for Newton's method for finite-dimensional zero-finding problems is a classical result for pointwise convex functions that are inverse monotonic (also called Collatz monotone \cite{collatz1952aufgaben}), cf., e.g., the book of Ortega and Rheinboldt \cite[Thm.~13.3.2]{ortega1970iterative}. Such problems arise, e.g., in solving non-linear elliptic PDEs. Roughly speaking, one might be tempted to say, that elliptic forward coefficient problems lead to inverse monotonic convex function, and
inverse elliptic coefficient problems lead to forward monotonic convex functions. 
Our extra criterion on the directional derivative evaluations allows us to write our forward monotonic function as an
affine transformation of an inverse monotonic function in a certain region and (together with some technical arguments to
ensure the iterates staying in this region), this is the major key in proving our global convergence result.

The paper is organized as follows. In section \ref{Sec:Newton}, we prove uniqueness, stability and global convergence of the Newton method
for continuously differentiable, pointwise convex and monotonic functions under a simple extra condition on the directional derivatives. 
In section~\ref{sect:inverse_bvp}, we apply this result to an inverse Robin coefficient problem, and show how to determine those measurements that uniquely and stably determine the unknown coefficient with a desired resolution via a globally convergent Newton iteration. We also give a numerical example in section~\ref{sect:inverse_bvp}. Throughout this paper, we take the somewhat lengthy, but hopefully reader-friendly approach of first presenting less technical intermediate results to motivate our approach.

\section{Uniqueness, stability and global Newton convergence}
\label{Sec:Newton}

We consider a continuously differentiable, pointwise convex and monotonic function
\[
F:\ U\subseteq \R^n\to \R^m
\]
where $n,m\in \N$, $m\geq n \geq 2$, and $U$ is a convex open set. In this section, we will derive a simple criterion that implies injectivity of $F$ on a multidimensional interval.
The criterion also allows us to estimate the Lipschitz stability constant of the left inverse $F^{-1}$ and, for $n=m$, ensures global convergence of Newton's method.  

\begin{remark}\label{remark:convex_mon}
Throughout this work, ''$\leq$'' is always understood pointwise for finite-dimensional vectors and matrices, 
and $x\not\leq y$ denotes the converse, i.e., that $x-y$ has at least one positive entry. 

Monotonicity and convexity are understood with respect to this pointwise partial order, i.e., $F:\ U\subseteq \R^n\to \R^m$ is monotonic if
\begin{align*}
x\leq y \quad \text{ implies } \quad F(x)\leq F(y) \quad \text{ for all } x,y\in U,
\end{align*}
and $F$ is convex if
\begin{align*}
F((1-t)x+ty)\leq (1-t)F(x)+tF(y) \quad \text{ for all } x,y\in U,\ t\in [0,1].
\end{align*}
We also say that a function $F$ is anti-monotone if $-F$ is monotonic.

For continuously differentiable functions, it is easily shown that monotonicity is equivalent to 
\begin{align}\label{eq:assumption_monotonic}
F'(x)y \geq 0  \quad  \text{ for all } x\in U,\ 0\leq y\in \R^n,
\end{align}
and thus equivalent to $F'(x)\geq 0$. Convexity is known to be equivalent to
\begin{alignat}{2}
\label{eq:assumption_convex} 
F(y)-F(x)&\geq F'(x)(y-x) \quad && \text{ for all } x,y\in U,
\end{alignat}
cf., e.g., \cite[Thm.~13.3.2]{ortega1970iterative}. 
All the proofs in this section use the monotonicity and convexity assumption in the form \eqref{eq:assumption_monotonic} and \eqref{eq:assumption_convex}.
\end{remark}

Throughout this work, we denote by $e_j\in \R^n$ the $j$-th unit vector in $\R^n$, $\1:=(1,1,\ldots,1)^T\in \R^n$, and $e_j':=\1-e_j$.  $I_n\in \R^{n\times n}$ denotes the $n$-dimensional identity matrix, and $\1 \1^T\in \R^{n\times n}$ is the matrix containing $1$ in all of its entries.

\subsection{A simple criterion for uniqueness and Lipschitz stability}\label{subsect:simple_uniqueness}

Before we state our result in its final form in subsection \ref{subsect:tight}, we derive two weaker (and less technical) results that motivate our
arguments and may be of independent interest. We first show a simple criterion that yields injectivity of $F$ and allows
us to estimate the Lipschitz stability constant of its left inverse $F^{-1}$.

\begin{theorem}\label{thm:uniqueness_simple}
Let $F:\ U\subseteq \R^n\to \R^m$, $m\geq n\geq 2$, be a continuously differentiable, pointwise convex and monotonic function on a convex open set $U$ containing $[-1,3]^n$.
If, additionally,
\begin{equation}\label{eq:thm_uniqueness_simple_assumption}
F'(-e_j+ 3e_j') \left(e_j - 3e_j'\right)\not\leq 0 \quad \text{ for all } j\in \{1,\ldots,n\},
\end{equation}
then the following holds:
\begin{enumerate}[(a)]
\item $F$ is injective on $[0,1]^n$.
\item $F'(x)\in \R^{m\times n}$ is injective for all $x\in [0,1]^n$.
\item With 
\begin{equation}\label{eq:thm_uniqueness_def_L}
L:=2 \left( \min_{j=1,\ldots,n}\ \max_{k=1,\ldots,m} e_k^T F'(-e_j+ 3e_j') \left(e_j - 3e_j'\right)\right)^{-1}>0,
\end{equation}
we have that for all $x,y\in [0,1]$
\[
\norm{x-y}_\infty\leq L \norm{F(x)-F(y)}_\infty, \quad \text{ and } \quad \norm{F'(x)^{-1}}_\infty \leq L,
\]
where $F'(x)^{-1}\in \R^{n\times m}$ denotes the left inverse of $F'(x)\in \R^{m\times n}$.
\end{enumerate}
\end{theorem}

To prove Theorem \ref{thm:uniqueness_simple}, we will first formulate an auxiliary lemma, which will also be used in our more technical results. Note that assumption \eqref{eq:max_lemma_aux} in the following lemma simply means that a row permutation of the non-negative matrix $F'(x)\in \R^{m\times n}$ is strictly diagonally dominant in its first $n$ rows.

\begin{lemma}\label{lemma:aux_lemma}
Let $F:\ U\subseteq \R^n\to \R^m$, $m\geq n\geq 2$, be continuously differentiable, pointwise convex and monotonic on a convex open set $U$, 
and let $L>0$. 
\begin{enumerate}[(a)]
\item If $x\in U$ fulfills 
\begin{equation}\label{eq:max_lemma_aux}
\max_{k=1,\ldots,m} e_k^T F'(x) \left(e_j - e_j'\right)\geq L^{-1} \quad \text{ for all } j\in \{1,\ldots,n\},
\end{equation}
then $F'(x)\in \R^{m\times n}$ is injective, and its left inverse fulfills $\norm{F'(x)}_\infty^{-1}\leq L$. 
\item If, additionally, $y\in U$ also fulfills \eqref{eq:max_lemma_aux}, then
\[
\norm{x-y}_\infty\leq L \norm{F(x)-F(y)}_\infty.
\]
\end{enumerate}
\end{lemma}
\begin{proof}
\begin{enumerate}[(a)]
\item For all $0\neq d \in \R^n$, at least one of the entries of $\frac{d}{\norm{d}_\infty}$ must be either $1$ or $-1$,
so that there exists $j\in \{1,\ldots,n\}$ with either
\[
\frac{d}{\norm{d}_\infty}\geq e_j-e_j', \quad \text{ or } \quad \frac{d}{\norm{d}_\infty}\leq -e_j+e_j'.
\]
We thus obtain from the monotonicity assumption \eqref{eq:assumption_monotonic} that either
\[
F'(x)\frac{d}{\norm{d}_\infty}\geq F'(x)(e_j-e_j'), \quad
\text{ or }  \quad F'(x)\frac{-d}{\norm{d}_\infty}\geq F'(x)(e_j-e_j').
\] 
In both cases, it follows from \eqref{eq:max_lemma_aux} that
\[
\frac{\norm{F'(x)d}_\infty}{\norm{d}_\infty}\geq \max_{k=1,\ldots,m} e_k^T F'(x) \left(e_j - e_j'\right)\geq L^{-1}.
\]
This proves injectivity of $F'(x)$ and the bound on its left inverse.

\item Likewise, for $x\neq y$, either
\[
\frac{x-y}{\norm{y-x}_\infty}\geq e_j-e_j', \quad \text{ or } \quad \frac{y-x}{\norm{x-y}_\infty}\geq e_j-e_j',
\]
so that by monotonicity \eqref{eq:assumption_monotonic} and assumption \eqref{eq:max_lemma_aux}, either
\[
\max_{k=1,\ldots,m} e_k^T F'(x)\frac{y-x}{\norm{x-y}_\infty}
\geq L^{-1},  \text{ or }  \max_{k=1,\ldots,m} e_k^T F'(y)\frac{x-y}{\norm{x-y}_\infty}
\geq L^{-1}.
\]
By convexity \eqref{eq:assumption_convex}, it then follows that
\begin{align*}
\norm{F(y)-F(x)}_\infty & = 
\max_{k=1,\ldots,m} |e_k^T (F(y)-F(x))|\\ & =
 \max_{k=1,\ldots,m} \max \{ e_k^T (F(y)-F(x)), e_k^T (F(x)-F(y)) \}\\
&\geq \max_{k=1,\ldots,m} \max \{ e_k^T F'(x)(y-x), e_k^T F'(y)(x-y) \}\\
&\geq \norm{y-x}_\infty L^{-1}.
\end{align*}
\hfill $\Box$
\end{enumerate}
\end{proof}

We can now prove Theorem~\ref{thm:uniqueness_simple} with lemma \ref{lemma:aux_lemma}.

\begin{proof}[of Theorem~\ref{thm:uniqueness_simple}]
Let $j\in \{1,\ldots,n\}$. Writing 
\[
z^{(j)}:=-e_j+ 3e_j'\in [-1,3]^n\subset U,
\]
we have that
\[
e_j - 3 e_j'  \leq x - z^{(j)}\leq 2e_j - 2e_j' \quad \text{ for all } x\in [0,1]^n.
\]
Thus we deduce from \eqref{eq:assumption_monotonic} and \eqref{eq:assumption_convex} that for all $x\in [0,1]^n$
\begin{align*}
F'(x) \left(e_j - e_j'\right)
&
\geq  \frac{1}{2} F'(x) \left(x - z^{(j)}\right)
\geq \frac{1}{2} \left( F(x) - F(z^{(j)})  \right)\\
&\geq \frac{1}{2} F'(z^{(j)}) \left(x - z^{(j)}\right)
\geq \frac{1}{2} F'(z^{(j)}) \left(e_j - 3e_j'\right).
\end{align*}
With the definition of $L$ in \eqref{eq:thm_uniqueness_def_L}, this shows that 
\begin{equation}\label{eq:thm_uniqueness_proof_aux}
\max_{k=1,\ldots,m} e_k^T F'(x) \left(e_j - e_j'\right)\geq L^{-1} \quad \text{ for all } x\in [0,1]^n,\ j\in \{1,\ldots,n\},
\end{equation}
so that (a), (b) and (c) follow from lemma \ref{lemma:aux_lemma}.
\hfill $\Box$
\end{proof}


\subsection{A simple criterion for global convergence of the Newton iteration}\label{subsect:simple_Newton}

We will now show how to ensure that a convex monotonic function $F$ has a unique zero, and that the
Newton method globally converges against this zero. 

\begin{theorem}\label{thm:simple_Newton}
Let $F:\ U\subseteq \R^n\to \R^n$, $n\geq 2$, 
be continuously differentiable, pointwise convex and monotonic on a convex open set $U$. If $[-2,n(n+3)]^n\subset U$, and
\begin{equation}\label{eq:thm_Newton_simple_assumption}
F'(z^{(j)})d^{(j)}\not\leq 0 \quad \text{ for all } j\in \{1,\ldots,n\},
\end{equation}
with 
\[
z^{(j)}:=-2 e_j + n(n+3) e_j',\quad \text{ and } \quad d^{(j)}:=e_j  - (n^2+3n+1)  e_j',
\]
then the following holds:
\begin{enumerate}[(a)]
\item $F$ is injective on $[-1,n]^n$, $F'(x)$ is invertible for all $x\in [-1,n]^n$, and for all $x,y\in [-1,n]^n$
\begin{equation}\label{eq:thm_simple_Newton_inverse_Lip}
\norm{x-y}_\infty\leq L \norm{F(x)-F(y)}_\infty, \quad \text{ and } \quad \norm{F'(x)^{-1}}_\infty \leq L,
\end{equation}
where 
\begin{equation}\label{eq:thm_Newton_def_L}
L:=(n+2) \left( \min_{j=1,\ldots,n}\ \max_{k=1,\ldots,n} e_k^T F'(z^{(j)})d^{(j)}\right)^{-1}>0.
\end{equation}
\item If, additionally, $F(0)\leq 0\leq F(\1)$, then there exists a unique 
\[
\hat x\in \left( \textstyle -\frac{1}{n-1},1+\frac{1}{n-1}\right)^n\quad \text{ with } \quad F(\hat x)=0,
\]
and this is the only zero of $F$ in $[-1,n]^n$.
The Newton iteration sequence
\begin{equation}\label{eq:Newton_iteration}
x^{(k+1)}:=x^{(k)}- F'(x^{(k)})^{-1}F(x^{(k)})\quad \text{with initial value $x^{(0)}:=\1$}
\end{equation}
is well defined (i.e., $F'(x^{(k)})$ is invertible in each step) and converges against $\hat x$.
Furthermore, for all $k\in \N$, $x^{(k)}\in (-1,n)^n$, and 
\[
0\leq M\hat x\leq M x^{(k+1)}\leq M x^{(k)}\leq Mx^{(0)}=(n+1)\1,
\]
where $M:=\1\1^T+ I_n\in \R^{n\times n}$. 
The rate of convergence of $x_k\to \hat x$ is superlinear. If $F'$ is locally Lipschitz, then the rate of convergence is quadratic.
\end{enumerate}
\end{theorem}
To prove Theorem \ref{thm:simple_Newton} we will first show the following lemma.

\begin{lemma}\label{lemma:simple_Newton}
Under the assumptions and with the notations of Theorem~\ref{thm:simple_Newton}, the following holds:
\begin{enumerate}[(a)]
\item For all $x\in [-1,n]^n$, 
\[
\max_{k=1,\ldots,n} e_k^T F'(x)(e_j-n e_j')\geq L^{-1}.
\]
\item $F$ is injective on $[-1,n]^n$, $F'(x)\in \R^{n\times n}$ is invertible for all $x\in [-1,n]^n$, and, for all $x,y\in [-1,n]^n$,
\[
\norm{x-y}_\infty\leq L \norm{F(x)-F(y)}_\infty, \quad \text{ and } \quad \norm{F'(x)^{-1}}_\infty \leq L.
\]
\item For all $x\in [-1,n]^n$, and all $0\neq y\in \R^n$ 
\[
F'(x)y\geq 0 \quad \text{ implies } \quad \max_{j=1,\ldots,n} y_j=\norm{y}_\infty \ \text{ and } \
\min_{j=1,\ldots,n} y_j> -\frac{1}{n}\norm{y}_\infty.
\]
\item $M$ is invertible, $M^{-1}=I_n-\frac{1}{n+1}\1 \1^T$. For all $x\in [-1,n]^n$, and $y\in \R^n$
\[
F'(x)y\geq 0 \quad \text{ implies } \quad My\geq 0,
\]
and thus $M F'(x)^{-1}\geq 0$.
\end{enumerate}
\end{lemma}
\begin{proof}
\begin{enumerate}[(a)]
\item Let $j\in \{1,\ldots,n\}$. Using $z^{(j)}=-2 e_j + n(n+3) e_j'$, we have that for all $x\in [-1,n]^n$ 
\begin{align*}
d^{(j)}=e_j  - (n^2+3n+1)  e_j'
\leq x-z^{(j)}\leq
(n+2)e_j - n(n+2)e_j'
\end{align*}
and thus it follows from monotonicity \eqref{eq:assumption_monotonic} and convexity \eqref{eq:assumption_convex} that
\begin{align*}
F'(x) \left(e_j - n e_j'\right)&=\frac{1}{n+2} F'(x) \left((n+2)e_j - n(n+2)e_j'\right)\\
 &\geq  \frac{1}{n+2} F'(x) \left(x - z^{(j)}\right)
\geq \frac{1}{n+2} \left( F(x) - F(z^{(j)})  \right)\\
&\geq \frac{1}{n+2} F'(z^{(j)}) \left(x - z^{(j)}\right)
\geq \frac{1}{n+2} F'(z^{(j)}) d^{(j)},
\end{align*}
which proves (a) using the definition of $L$ in \eqref{eq:thm_Newton_def_L}.
\item Since (a) implies a fortiori that for all $x\in [-1,n]^n$ 
\[
\max_{k=1,\ldots,n} e_k^T F'(x)(e_j- e_j')\geq L^{-1},
\]
the assertion (b) follows from lemma~\ref{lemma:aux_lemma}.
\item Let $x\in [-1,n]^n$, and $0\neq y\in \R^n$. If there exists an index $j\in \{1,\ldots,n\}$ with $y_j\leq -\frac{1}{n}\norm{y}_\infty$, then $y\leq -\frac{1}{n}\norm{y}_\infty e_j + \norm{y}_\infty e_j'$, so that
\[
- \frac{n}{\norm{y}_\infty } y  \geq e_j -n e_j', \quad \text{ and thus, by (a),} \quad - \frac{n}{\norm{y}_\infty } F'(x) y\not\leq 0.
\] 
By contraposition, this shows that 
\[
F'(x)y\geq 0 \quad \text{ implies } \quad \min_{j=1,\ldots,n} y_j> -\frac{1}{n}\norm{y}_\infty,
\]
which also shows that $\norm{y}_\infty= \max_{j=1,\ldots,n} y_j$.
\item It is easily checked that
\begin{align*}
\left(I_n-\frac{1}{n+1}\1 \1^T \right) \left(\1 \1^T+I_n\right)
=I_n,
\end{align*}
which shows that $M$ is invertible and $M^{-1}=I_n-\frac{1}{n+1}\1 \1^T$

Moreover, using (c) it follows that $F'(x)y\geq 0$ implies that for all $k\in \N$, 
\[
\sum_{j=1}^n y_j + y_k\geq \max_{j=1,\ldots,n} y_j + n \min_{j=1,\ldots,n} y_j\geq 0.
\]
so that $F'(x)y\geq 0$ implies $My=(\1 \1^T+I_n)y\geq 0$. This also shows that
\[
F'(x) F'(x)^{-1} y=y\geq 0 \quad \text{ implies } \quad M F'(x)^{-1} y\geq 0,
\]
and thus $M F'(x)^{-1}\geq 0$.
\hfill $\Box$
\end{enumerate}
\end{proof}

Note that by lemma~\ref{lemma:simple_Newton}(d), $\tilde F(x):=F(M^{-1}x)$ is a convex function with Collatz monotone derivative \cite{collatz1952aufgaben}, i.e. $\tilde F'(x)^{-1}= M F'(M^{-1} x)^{-1}\geq 0$. If the Newton iterates do not leave the region where convexity and Collatz monotony holds, then classical results on monotone Newton methods (cf., e.g., Ortega and Rheinboldt \cite[Thm.~13.3.4]{ortega1970iterative}) yield global Newton convergence for $\tilde F$, and thus for $F$ since the Newton method is invariant under linear transformation. The following lemma utilizes some arguments from the classical result \cite[Thm.~13.3.4]{ortega1970iterative} for our situation. 

\begin{lemma}\label{lemma:aux_Newton_monotone}
Let $F:\ U\subseteq \R^n\to \R^n$ be continuously differentiable and pointwise convex on a convex open set $U$ containing zero, and let $M\in \R^{n\times n}$. 
We assume that for some point $x\in U$, $F'(x)\in \R^{n\times n}$ is invertible, 
\begin{equation}
\label{eq:aux_Newton_monotone_ass}
MF'(x)^{-1}\geq 0,\quad F(x)\geq 0\geq F(0), \quad \text{ and } \quad Mx\geq 0.
\end{equation}
Then for all $t\in [0,1]$
\[
x^{(t)}:=x- t F'(x)^{-1} F(x) 
\]
fulfills $0\leq Mx^{(t)}\leq Mx$. Moreover, if $x^{(t)}\in U$ then $F(x^{(t)})\geq 0$.
\end{lemma}
\begin{proof}
The assumptions \eqref{eq:aux_Newton_monotone_ass} and the convexity \eqref{eq:assumption_convex} yield that for all $t\in [0,1]$
\begin{align*}
0&\leq - t M F'(x)^{-1} F(0)\\
&= Mx^{(t)} - \left(M x- tM F'(x)^{-1} F(x)\right) - tM F'(x)^{-1}F(0)\\
&= Mx^{(t)} - M x + t M F'(x)^{-1} \left( F(x) - F(0)\right)\\
&\leq Mx^{(t)} - M x + tM F'(x)^{-1} F'(x) ( x - 0)=Mx^{(t)}-(1-t)Mx.
\end{align*}
Moreover, 
\[
Mx-Mx^{(t)}=t M F'(x)^{-1} F(x)\geq 0,
\]
so that $Mx\geq Mx^{(t)}\geq (1-t)Mx\geq 0$ is proven.

If $x^{(t)}\in U$, then we also obtain from the convexity assumption \eqref{eq:assumption_convex} that
\begin{align*}
F(x^{(t)})-F(x)\geq F'(x)( x^{(t)} - x)=-t F(x),
\end{align*}
which shows that $F(x^{(t)})\geq (1-t)F(x)\geq 0$.\hfill $\Box$
\end{proof}

\begin{proof}[of Theorem \ref{thm:simple_Newton}]
Theorem \ref{thm:simple_Newton}(a) has been proven in lemma~\ref{lemma:simple_Newton}(b).

To prove (b), let $x^{(k)}\in (-1,n)^n$ with $F(x^{(k)})\geq 0$ and $0\leq M x^{(k)}\leq M\1$. Then, by (a), $F'(x^{(k)})\in \R^{n\times n}$ is invertible, so that 
\[
x^{(k+t)}:=x^{(k)}- t F'(x^{(k)})^{-1}F(x^{(k)})\in \R^n, \quad t\in [0,1],
\]
is well defined. 

We will prove that $x^{(k+1)}\in (-1,n)^n$. We argue by contradiction, and assume that this is not the case. Then, by continuity, there exists $t\in (0,1]$ with $x^{(k+t)}\in [-1,n]^n\setminus (-1,n)^n$ and, by lemma \ref{lemma:aux_Newton_monotone},
\begin{equation}\label{eq:simple_Newton_Mxt}
F(x^{(k+t)})\geq 0 \quad \text{ and } \quad 0\leq M x^{(k+t)}\leq M x^{(k)}\leq M\1.
\end{equation}
Convexity \eqref{eq:assumption_convex} then yields that
\begin{align*}
F'(x^{(k+t)})(x^{(k+t)}-0)\geq
F(x^{(k+t)})-F(0)\geq 0,
\end{align*}
and using lemma~\ref{lemma:simple_Newton}(c) this would imply that
\begin{align}\label{eq:min_xkplust}
\min_{j=1,\ldots,n}\, x^{(k+t)}_j&>-\frac{1}{n}\, \max_{j=1,\ldots,n}\, x^{(k+t)}_j \geq -1.
\end{align}
For all $l\in \{1,\ldots,n\}$, we obtain from \eqref{eq:simple_Newton_Mxt} and \eqref{eq:min_xkplust}
\begin{align*}
 n+1&= e_l^T M\1 \geq e_l^T M x^{(k+t)}=x^{(k+t)}_l + \sum_{j=1}^n x^{(k+t)}_j
 > 2 x^{(k+t)}_l - (n-1).
\end{align*}
Hence, $\max_{j=1,\ldots,n}\, x^{(k+t)}_j< n$, so that $x^{(k+t)}\in (-1,n)^n$.
Since this contradicts $x^{(k+t)}\in [-1,n]^n\setminus (-1,n)^n$, we have proven that $x^{(k+1)}\in (-1,n)^n$.

Using lemma \ref{lemma:aux_Newton_monotone} again, this shows that for all 
\[
x^{(k)}\in (-1,n)^n\ \text{ with } F(x^{(k)})\geq 0, \ \text{ and } \ 0\leq M x^{(k)}\leq M\1,
\]
the next Newton iterate $x^{(k+1)}$ is well-defined and also fulfills 
\[
x^{(k+1)}\in (-1,n)^n\ \text{ with } F(x^{(k+1)})\geq 0, \ \text{ and } \ 0\leq M x^{(k+1)}\leq M x^{(k)}\leq M\1.
\]
Hence, for $x^{(0)}:=\1$, the Newton algorithm produces a well-defined 
sequence $x^{(k)}\in (-1,n)^n$ for which $Mx^{(k)}$ is monotonically non-increasing and bounded. Hence, $(Mx^{(k)})_{k\in \N}$ and thus also $(x^{(k)})_{k\in \N}$ converge.
We define
\[
\hat x:=\lim_{k\to \infty} x^{(k)}\in [-1,n]^n.
\]
Since $F$ is continuously differentiable and $F'(\hat x)$ is invertible, it follows from the Newton iteration formula \eqref{eq:Newton_iteration} that $F(\hat x)=0$.  
Also, the monotone convergence of $(Mx^{(k)})_{k\in \N}$ shows that
\[
0\leq M\hat x \leq Mx^{(k)} \leq M\1 \quad \text{ for all } k\in \N.
\]
Moreover, since this is the standard Newton iteration, the convergence speed is superlinear and the speed is quadratic if $F'$ is Lipschitz
continuous in a neighbourhood of $\hat x$. 

It only remains to show that $\hat x\in (-\frac{1}{n-1},\frac{n}{n-1})^n\subset (-1,2)^n$. For this, we use the convexity to obtain
\begin{align*}
F'(\hat x)(\hat x-0)&\geq F(\hat x)-F(0)\geq 0,\quad \\\
F'(\1)(\1-\hat x)&\geq F(1)-F(\hat x)\geq 0,
\end{align*}
which then implies by lemma~\ref{lemma:simple_Newton}(c) that
\begin{align*}
\min_{j=1,\ldots,n}\, \hat x_j &> -\frac{1}{n} \max_{j=1,\ldots,n}\, \hat x_j,\\
\min_{j=1,\ldots,n}\, (1-\hat x_j) &> -\frac{1}{n} \max_{j=1,\ldots,n}\, (1-\hat x_j).
\end{align*}
From this we obtain that
\begin{align}
\label{eq_simple_newton_xhat_bounds}
\min_{j=1,\ldots,n}\, \hat x_j &> -\frac{1}{n} \max_{j=1,\ldots,n}\, \hat x_j
=\frac{1}{n} \min_{j=1,\ldots,n}\, (1-\hat x_j)-\frac{1}{n}\\
\nonumber
&>-\frac{1}{n^2} \max_{j=1,\ldots,n}\, (1-\hat x_j)-\frac{1}{n}
= \frac{1}{n^2} \min_{j=1,\ldots,n}\, \hat x_j -\frac{1}{n^2} - \frac{1}{n},
\end{align}
which yields
$\min_{j=1,\ldots,n}\, \hat x_j> -\frac{1}{n-1}$. Using \eqref{eq_simple_newton_xhat_bounds} again, we then obtain
\begin{align*}
-\frac{1}{n} \max_{j=1,\ldots,n}\, \hat x_j &> \frac{1}{n^2} \min_{j=1,\ldots,n}\, \hat x_j -\frac{1}{n^2} - \frac{1}{n}
>-\frac{1}{n-1},
\end{align*}
which shows $\max_{j=1,\ldots,n}\, \hat x_j<\frac{n}{n-1}$.
\hfill $\Box$
\end{proof}


\subsection{A result with tighter domain assumptions}\label{subsect:tight}

Our results in subsections \ref{subsect:simple_uniqueness} and \ref{subsect:simple_Newton}
require the considered function $F$ to be defined (and convex and monotonic) on a much larger set than $[0,1]^n$.
For some applications (such as the inverse coefficient problem in section \ref{sect:inverse_bvp}), the following more technical variant of Theorem \ref{thm:simple_Newton} is useful, as it allows us treat the case where the domain of definition is an arbitrarily small neighbourhood of $[0,1]^n$.

\begin{theorem}\label{thm:tight_Newton}
Let $\epsilon>0$ and $c\geq 2+\frac{2}{\epsilon}$. 
Let $F:\ U\subseteq \R^n\to \R^n$, $n\geq 2$, 
be continuously differentiable, pointwise convex and monotonic on a convex open set $U$.
If $[-\frac{1+\epsilon}{cn}-\frac{\epsilon}{2cn},1+2\epsilon]^n\subset U$, and
\begin{equation}\label{eq:thm_Newton_tight_assumption}
F'(z^{(j,k)})d^{(j)}\not\leq 0 \quad \text{ for all } j\in \{1,\ldots,n\},\ k=\{1,\ldots,K\},
\end{equation}
where $K:=\operatorname{ceil}(\frac{2cn}{\epsilon}\left(1+\epsilon + \frac{1+\epsilon}{cn}\right))\in \N$,
\begin{align}
\label{eq:Newton_tight_def_z}
z^{(j,k)}&:=\left( -\frac{1+\epsilon}{cn}+(k-2)\frac{\epsilon}{2cn} \right) e_j + \left( 1+2\epsilon  \right)e_j',\\
\label{eq:Newton_tight_def_d}
d^{(j)}&:=\frac{1}{2} e_j  - \frac{1+\epsilon+cn+2\epsilon cn}{\epsilon} e_j'.
\end{align}
%
%
then the following holds on $O:=(-\frac{1+\epsilon}{cn},1+\epsilon)^n\supset [0,1]^n$.
\begin{enumerate}[(a)]
\item $F$ is injective on $\overline O$. For all $x,y\in \overline O$, $F'(x)$ is invertible and
\[
\norm{x-y}_\infty\leq L \norm{F(x)-F(y)}_\infty, \quad \text{ and } \quad \norm{F'(x)^{-1}}_\infty \leq L,
\]
where 
\begin{equation}
L:= \left( \min_{\substack{j=1,\ldots,n,\\ k=1,\ldots,K}}\ \max_{l=1,\ldots,n} e_l^T F'(z^{(j,k)})d^{(j)}\right)^{-1}>0.
\end{equation}
\item If, additionally, $F(0)\leq 0\leq F(\1)$, then there exists a unique 
\[
\hat x\in \overline O \quad \text{ with } \quad F(\hat x)=0.
\]
The Newton iteration sequence
\begin{equation}
x^{(i+1)}:=x^{(i)}- F'(x^{(i)})^{-1}F(x^{(i)})\quad \text{with initial value $x^{(0)}:=\1$}
\end{equation}
is well defined (i.e., $F'(x^{(i)})$ is invertible in each step) and converges against $\hat x$.
For all $i\in \N$
\begin{align*}
x^{(i)}&\in O , \quad \text{ and } \quad
0\leq M\hat x\leq M x^{(i+1)}\leq M x^{(i)}\leq Mx^{(0)}=(1+cn)\1,
\end{align*}
where $M:=\1 \1^T+(1+(c-1)n)I_n\in \R^{n\times n}$. 
The rate of convergence of $x_i\to \hat x$ is superlinear. If $F'$ is locally Lipschitz then the rate of convergence is quadratic.
\end{enumerate}
\end{theorem}

To prove Theorem~\ref{thm:tight_Newton} we first prove a variant of lemma \ref{lemma:simple_Newton}
with tighter domain assumptions.

\begin{lemma}\label{lemma:tight_Newton}
Under the assumptions and with the notations of Theorem~\ref{thm:tight_Newton}, the following holds:
\begin{enumerate}[(a)]
\item For all $x\in \overline O$, and $j\in \{1,\ldots,n\}$,
\[
\max_{l=1,\ldots,n} e_l^T F'(x)(e_j-c n e_j')\geq L^{-1}.
\]
\item $F$ is injective on $\overline O$.  For all $x,y\in \overline O$, $F'(x)$ is invertible, 
\[
\norm{x-y}_\infty\leq L \norm{F(x)-F(y)}_\infty, \quad \text{ and } \quad \norm{F'(x)^{-1}}_\infty \leq L.
\]
\item For all $x\in \overline O$, and $0\neq y\in \R^n$ 
\[
F'(x)y\geq 0 \ \text{ implies } \ \max_{j=1,\ldots,n} y_j=\norm{y}_\infty \ \text{ and } \
\min_{j=1,\ldots,n} y_j> -\frac{1}{c n}\norm{y}_\infty.
\]
\item $M$ is invertible, $M^{-1}=\frac{1}{1+(c-1)n} \left( I_n- \frac{1}{1+cn} \1 \1^T \right)$. For all $x\in \overline O$, and $y\in \R^n$
\[
F'(x)y\geq 0 \quad \text{ implies } \quad My\geq 0,
\]
and thus $M F'(x)^{-1}\geq 0$.
\end{enumerate}
\end{lemma}
\begin{proof} 
We use the same arguments as in lemma \ref{lemma:simple_Newton}.
To prove (a) let $j\in \{1,\ldots,n\}$ and $x\in \overline O=[-\frac{1+\epsilon}{cn},1+\epsilon]^n$.
Then, by the definition of $K$, there exists $k\in \{1,\ldots,K\}$, so that
\[
-\frac{1+\epsilon}{cn}+(k-1)\frac{\epsilon}{2cn} \leq x_j\leq
-\frac{1+\epsilon}{cn}+k\frac{\epsilon}{2cn}.
\]
It follows from the definition of $z^{(j,k)}$ and $d^{(j)}$ in \eqref{eq:Newton_tight_def_z} and \eqref{eq:Newton_tight_def_d} that
\begin{align*}
x-z^{(j,k)} &\geq 
\frac{\epsilon}{2cn} e_j  - \left(\frac{1+\epsilon}{cn}+1+2\epsilon\right) e_j'
=\frac{\epsilon}{cn} d^{(j)},\\
x-z^{(j,k)} &\leq 
\frac{\epsilon}{cn} e_j -\epsilon e_j'.
\end{align*}
We thus obtain
\begin{align*}
\lefteqn{F'(x) \left(e_j - cn e_j'\right)
=\frac{cn}{\epsilon} F'(x) \left( \frac{\epsilon}{cn} e_j - \epsilon e_j'\right)
\geq  \frac{cn}{\epsilon} F'(x) \left(x-z^{(j,k)}\right)}\\
&\geq \frac{cn}{\epsilon} \left( F(x) - F(z^{(j,k)})  \right)
\geq \frac{cn}{\epsilon} F'(z^{(j,k)}) \left(x - z^{(j,k)}\right)
\geq F'(z^{(j,k)}) d^{(j)} ,
\end{align*}
which proves (a). Since this also implies a fortiori that 
\[
\max_{l=1,\ldots,n} e_l^T F'(x)(e_j- e_j')\geq L^{-1} \quad \text{ for all $x\in \overline O$,}
\]
(b) follows from lemma~\ref{lemma:aux_lemma}.

To show (c) by contraposition, let $x\in \overline O$, $0\neq y\in \R^n$, and  assume that for some index $j\in \{1,\ldots,n\}$,
we have that $y_j\leq -\frac{1}{cn}\norm{y}_\infty$. Then $y\leq -\frac{1}{cn}\norm{y}_\infty e_j + \norm{y}_\infty e_j'$, so that
\[
- \frac{cn}{\norm{y}_\infty } y  \geq e_j -cn e_j', \quad \text{ and thus, by (a),} \quad - \frac{cn}{\norm{y}_\infty } F'(x) y\not\leq 0.
\] 
By contraposition, this shows that 
\[
F'(x)y\geq 0 \quad \text{ implies } \quad \min_{j=1,\ldots,n} y_j> -\frac{1}{cn}\norm{y}_\infty,
\]
and the latter also implies that $\norm{y}_\infty= \max_{j=1,\ldots,n} y_j$.

For the proof of (d), it is easily checked that
\begin{align*}
\frac{1}{1+(c-1)n} \left( I_n- \frac{1}{1+cn} \1 \1^T \right) \left( \1 \1^T+(1+(c-1)n)I_n \right)
=I_n,
\end{align*}
which shows the invertibility of $M$ and the asserted formula for $M^{-1}$.
Moreover, for all $y\in \R^n$, and $l\in \{1,\ldots,n\}$, 
\begin{align*}
e_l^T M y &= \sum_{j=1}^n y_j + (1+(c-1)n)y_l\\
&\geq \max_{j=1,\ldots,n} y_j + (n-1) \min_{j=1,\ldots,n} y_j
+ (1+(c-1)n)y_l\\
&=\max_{j=1,\ldots,n} y_j + cn \min_{j=1,\ldots,n} y_j.
\end{align*}
Hence, using (c), for all $x\in \overline O$, $F'(x)y\geq 0$ implies $My\geq 0$.
As this also shows that
\[
F'(x) F'(x)^{-1} y=y\geq 0 \quad \text{ implies } \quad M F'(x)^{-1} y\geq 0,
\]
we have $M F'(x)^{-1}\geq 0$.
\hfill $\Box$
\end{proof}

\begin{proof}[of Theorem \ref{thm:tight_Newton}]
We proceed as in the proof of Theorem~\ref{thm:simple_Newton}.
Assertion (a) has already been proven in lemma~\ref{lemma:tight_Newton}(b).

To prove (b), let $x^{(i)}\in O$ with $F(x^{(i)})\geq 0$ and $0\leq M x^{(i)}\leq M\1$. Then, by (a), $F'(x^{(i)})\in \R^{n\times n}$ is invertible, so that 
\[
x^{(i+t)}:=x^{(i)}- t F'(x^{(i)})^{-1}F(x^{(i)})\in \R^n, \quad t\in [0,1]
\]
is well defined.

We will prove that $x^{(i+1)}\in O=(-\frac{1+\epsilon}{cn},1+\epsilon)^n$. We argue by contradiction, and assume that this is not the case. Then, by continuity, there exists $t\in (0,1]$ with $x^{(i+t)}\in \overline O\setminus O$, so that, by lemma \ref{lemma:aux_Newton_monotone},
\begin{equation}\label{eq:tight_Newton_Mxt}
F(x^{(i+t)})\geq 0 \quad \text{ and } \quad 0\leq M x^{(i+t)}\leq M x^{(i)}\leq M\1.
\end{equation}
Convexity \eqref{eq:assumption_convex} then yields that
\begin{align*}
F'(x^{(i+t)})(x^{(i+t)}-0)\geq
F(x^{(i+t)})-F(0)\geq 0,
\end{align*}
and using lemma~\ref{lemma:tight_Newton}(c) this would imply
\begin{align}\label{eq:tight_min_xkplust}
\min_{j=1,\ldots,n}\, x^{(i+t)}_j&>-\frac{1}{cn}\, \max_{j=1,\ldots,n}\, x^{(i+t)}_j \geq -\frac{1+\epsilon}{cn}.
\end{align}
For all $l\in \{1,\ldots,n\}$, we obtain from \eqref{eq:tight_Newton_Mxt} and \eqref{eq:tight_min_xkplust}
\begin{align*}
1+cn&= e_l^T M\1\geq e_l^T M x^{(i+t)} 
= (1 + (c-1)n)x^{(i+t)}_l + \sum_{j=1}^n x^{(i+t)}_j\\
&\geq (2 + (c-1)n)x^{(i+t)}_l -(n-1)\frac{1+\epsilon}{cn},
\end{align*}
and thus
\begin{align*}
x^{(i+t)}_l&\leq \frac{1+cn + (n-1)\frac{1+\epsilon}{cn}}{2 + (c-1)n}
=  1+ \frac{(n-1)(1+\epsilon+cn)}{(2 + (c-1)n)cn}.
\end{align*}
An elementary computation shows that 
\begin{equation*}
\frac{1+\epsilon + cn}{(2 + (c-1)n)cn}<
\frac{(1+\epsilon)cn + cn}{(2 + (c-1)n)cn}
= \frac{2+\epsilon}{2 + (c-1)n}
< \frac{2+\epsilon}{(c-1)n}
\leq  \frac{\epsilon}{n},
\end{equation*}
where we used $cn> 1$ for the first inequality, and we used the assumption  $c\geq 2+\frac{2}{\epsilon}=\frac{2+\epsilon}{\epsilon}+1$ for the last inequality. 
Hence, for all $l\in \{1,\ldots,n\}$,
\[
x^{(i+t)}_l<1+\epsilon\frac{n-1}{n}<1+\epsilon, \quad \text{ so that } \quad x^{(i+t)}\in O.
\]
This contradicts $x^{(i+t)}\in \overline O\setminus O$, and thus shows that $x^{(i+1)}\in O$.

Using lemma \ref{lemma:aux_Newton_monotone} again, this shows that for all 
\[
x^{(i)}\in O\ \text{ with } F(x^{(i)})\geq 0, \ \text{ and } \ 0\leq M x^{(i)}\leq M\1,
\]
the next Newton iterate $x^{(i+1)}$ is well-defined and also fulfills 
\[
x^{(i+1)}\in O\ \text{ with } F(x^{(i+1)})\geq 0, \ \text{ and } \ 0\leq M x^{(i+1)}\leq M x^{(i)}\leq M\1.
\]
Hence, for $x^{(0)}:=\1$, the Newton algorithm produces a well-defined 
sequence $x^{(i)}\in O$ for which $Mx^{(i)}$ is monotonically non-increasing and bounded. Hence, $(Mx^{(i)})_{i\in \N}$ and thus also $(x^{(i)})_{k\in \N}$ converge.
We define
\[
\hat x:=\lim_{i\to \infty} x^{(i)}\in \overline O.
\]
Since $F$ is continuously differentiable and $F'(\hat x)$ is invertible, it follows from the Newton iteration formula that $F(\hat x)=0$.  
Also, the monotone convergence of $(Mx^{(i)})_{i\in \N}$ shows that
\[
0\leq M\hat x \leq Mx^{(i)} \leq M\1 \quad \text{ for all } i\in \N.
\]
Moreover, since this is the standard Newton iteration, the convergence speed is superlinear and the speed is quadratic if $F'$ is Lipschitz
continuous in a neighbourhood of $\hat x$. 
\hfill $\Box$
\end{proof}

\section{Application to an inverse Robin transmission problem}\label{sect:inverse_bvp}

We will now show how to use our result to obtain 
uniqueness, stability and global convergence results for an inverse coefficient problem with finitely many measurements.
More precisely, we show how to choose finitely many measurements so that they uniquely determine the unknown coefficient function with a given resolution,
Lipschitz stability holds, and Newton's method globally converges.

\subsection{The setting}
We consider the inverse Robin transmission problem for the Laplace equation from \cite{harrach2019global}, that is motivated by corrosion detection. 
Note that similar problems have also been studied for the Helmholtz equation under the name conductive boundary condition or transition boundary condition
\cite{lyalinov1998transition,angell1990resistive}. We first introduce the idealized infinite-di\-men\-sion\-al forward and inverse problem following \cite{harrach2019global} and then study the case of finitely many measurements.

\subsubsection{The infinite-dimensional forward and inverse problem}

Let $\Omega \subset \R^d$ ($d\geq 2$), be a bounded domain with Lipschitz boundary $\partial \Omega$  and let $D$ be an open subset of $\Omega$, with $\overline D\subset \Omega$, Lipschitz boundary $\Gamma:=\partial D$ and connected complement $\Omega\setminus\overline{D}$, cf. figure \ref{fig:Setting} in the numerical section for a sketch of the setting.

We assume that $\Omega$ describes an electrically conductive imaging domain, with a-priori known conductivity that we set to $1$ 
for the ease of presentation. (Note that all of the following results remain valid if the conductivity in $D$ and in $\Omega\setminus\overline{D}$
are a-priori known spatially dependent functions as long as they are sufficiently regular to allow unique continuation arguments). We assume that corrosion effects on the interface $\Gamma$ can be modelled with an unknown Robin transmission parameter $\gamma\in L^\infty_+(\Gamma)$, where $L^\infty_+$ denotes the subset of $L^\infty$-functions with positive essential infima. 

Applying an electrical current flux $g\in L^2(\partial\Omega)$ on the boundary $\partial \Omega$ then yields an electric potential
$u\in H^1(\Omega)$ solving the following Robin transmission problem with Neumann boundary values
\begin{alignat}{2}
\label{eq:Robin1} \Delta u &=0 \quad && \text{ in } \Omega\setminus \Gamma,\\
\label{eq:Robin2} \partial_{\nu} u|_{\partial \Omega}&= g\quad && \text{ on  }\partial \Omega,\\
\label{eq:Robin3} \llbracket u \rrbracket_\Gamma&=  0 \quad && \text{ on }  \Gamma,\\
\label{eq:Robin4} \llbracket \partial_{\nu} u\rrbracket_\Gamma &= \gamma u \quad &&\text{ on }  \Gamma,
\end{alignat}
where  $\nu$ is the  unit normal vector to the  interface  $\Gamma$ or $\partial\Omega$ pointing outward of $D$, resp., $\Omega$,
and  
\[
\llbracket u \rrbracket:= u^+|_{\Gamma}-u^-|_{\Gamma}, 
\quad \text{ and } \quad 
\llbracket \partial_{\nu}u \rrbracket:=\partial_{\nu}u^+|_\Gamma - \partial_{\nu}u^-|_\Gamma,
\]
denote the jump of the Dirichlet, resp., Neumann trace values on $\Gamma$, with the superscript ''$+$'' denoting that
the trace is taken from $\Omega\setminus D$ and ''$-$'' denoting the trace taken from $D$.
In the following, we often denote the solution of \eqref{eq:Robin1}--\eqref{eq:Robin4} by $u_\gamma^{(g)}$ to point out its dependence on the Robin transmission coefficient $\gamma$ and the Neumann boundary data $g$. 

It is easily seen that this problem is equivalent to the variational formulation of finding 
$u_\gamma^{(g)}\in H^1(\Omega)$ such that
\begin{equation}
\label{Robin_variational}
\int_\Omega \nabla u_\gamma^{(g)}\cdot\nabla w\dx+\int_\Gamma \gamma u_\gamma^{(g)} w \dx[s]=\int_{\partial\Omega}gw\dx[s] \quad  \text{ for all } w\in  H^1(\Omega),
\end{equation}
and that \eqref{Robin_variational} is uniquely solvable by the Lax-Milgram-Theorem. Hence, we can define the
Neumann-to-Dirichlet map
\[
\Lambda(\gamma):\ L^2(\partial \Omega)\to  L^2(\partial \Omega), \ g\mapsto u_\gamma^{(g)}|_{\partial \Omega},
\ \text{ where $u_\gamma^{(g)}\in H^1(\Omega)$ solves \eqref{Robin_variational}}.
\]
It is easy to show that $\Lambda(\gamma)$ is a compact self-adjoint linear operator. 

One may regard $\Lambda(\gamma)$ as an idealized model (the so-called continuum model) 
of all electric current/voltage measurements that can be carried out on the outer boundary $\partial \Omega$.
Hence the infinite-dimensional inverse coefficient problem of determining a Robin transmission coefficient from boundary measurements can
be formulated as the problem to
\[
\text{ reconstruct } \quad \gamma\in L^\infty_+(\Gamma) \quad \text{ from } \quad \Lambda(\gamma)\in \LL(L^2(\partial \Omega)).
\]

We summarize some more properties of the infinite-dimensional forward mapping $\Lambda$ in the following lemma.

\begin{lemma}\label{lemma:Lambda_properties}
\begin{enumerate}[(a)]
\item The non-linear mapping
\[
\Lambda:\ L^\infty_+(\Gamma)\to \LL(L^2(\partial \Omega)), \quad \gamma\mapsto \Lambda(\gamma)
\]
is Fr\'echet differentiable. Its derivative 
\[
\Lambda':\ L^\infty_+(\Gamma)\to \LL(L^\infty(\Gamma),\LL(L^2(\partial \Omega)))
\]
is given by the bilinear form 
\begin{equation}\label{eq:Lambda_derivative}
\int_{\partial \Omega} g  \left(\Lambda'(\gamma)\delta\right) h\dx[s]  = -\int_\Gamma \delta u_\gamma^{(g)}u_\gamma^{(h)} \dx[s],
\end{equation}
for all $\gamma\in L^\infty_+(\Gamma)$, $\delta \in L^\infty(\Gamma)$, and $g,h\in L^2(\partial \Omega)$,
where $u_\gamma^{(g)}\in H^1(\Omega)$ solves \eqref{Robin_variational}.
$\Lambda'$ is locally Lipschitz continuous and $\Lambda'(\gamma)\delta\in \LL(L^2(\partial \Omega))$ is compact and self-adjoint.
\item For all $g\in L^2(\partial \Omega)$ and all $\gamma_1,\gamma_2\in L^\infty_+(\Omega)$,
\begin{equation*}
\int_{\partial \Omega} g \left( \Lambda(\gamma_2)-\Lambda(\gamma_1) \right) g \dx[s] \geq
\int_{\partial \Omega} g \left( \Lambda'(\gamma_1) (\gamma_2-\gamma_1)\right) g \dx[s].
\end{equation*}
\item For all $\gamma\in L^\infty_+(\Omega)$, $\delta\in L^\infty(\Omega)$, and $g\in L^2(\partial \Omega)$,
\begin{equation*}
\delta(x)\geq 0 \text{ for $x\in \Omega$ a.e.} \quad \text{ implies } \quad \int_{\partial \Omega} g \left(  \Lambda'(\gamma)\delta \right) g \dx[s]\leq 0.
\end{equation*}
\end{enumerate}
\end{lemma}
\begin{proof}
Obviously, for all $\gamma\in L^\infty_+(\Gamma)$ and $\delta \in L^\infty(\Gamma)$, \eqref{eq:Lambda_derivative} defines a compact self-adjoint linear operator $\Lambda'(\gamma)\delta\in  \LL(L^2(\partial \Omega))$. Moreover, it follows from the monotonicity estimate in \cite[Lemma 4.1]{harrach2019global} that for all $\delta \in L^\infty(\Gamma)$ (that are sufficiently small so that
$\gamma+\delta\in L^\infty_+(\Gamma)$)
\[
\int_\Gamma \delta |u_\gamma^{(g)}|^2 \dx[s]
\geq  \int_{\partial \Omega} g \left(\Lambda(\gamma)-\Lambda(\gamma+\delta)\right)\dx[s]
\geq \int_\Gamma \left( \gamma-\frac{\gamma^2}{\gamma+\delta}  \right) |u_\gamma^{(g)}|^2 \dx[s]
\]
and thus
\begin{align*}
\lefteqn{\norm{\Lambda(\gamma+\delta)-\Lambda(\gamma)-\Lambda'(\gamma)\delta}_{\LL(L^2(\partial \Omega))}}\\
&=\sup_{g\in L^2(\partial \Omega)} \left| \int_{\partial \Omega} g \left(\Lambda(\gamma+\delta)-\Lambda(\gamma)-\Lambda'(\gamma)\delta \right)\dx[s]\right|
 \leq  \int_\Gamma \left( \frac{\delta^2  }{\gamma+\delta} \right) |u_\gamma^{(g)}|^2 \dx[s]\\
&= O(\norm{\delta}^2).
\end{align*}
This shows that $\Lambda$ is Fr\'echet differentiable for all $\gamma\in L^\infty_+(\Gamma)$, and that $\Lambda'(\gamma)$ is its derivative. Since it is easily shown that $u_\gamma^{(g)}\in H^1(\Omega)$ depends locally Lipschitz continuously on $\gamma\in L^\infty_+(\Gamma)$, it also follows that $\Lambda'$ is locally Lipschitz continuous. This proves (a). 

(b) is shown in \cite[Lemma 4.1]{harrach2019global}, and (c) follows from \eqref{eq:Lambda_derivative}.
\end{proof}

Note that lemma \ref{lemma:Lambda_properties} shows that $\Lambda$ is a convex, anti-monotone function with respect to the pointwise partial order on $L^\infty_+(\Omega)$, and the Loewner partial order in the space of compact self-adjoint operators on $L^2(\partial \Omega)$. 
These properties are the key to formulate the finite-dimensional inverse problem as a zero finding problem for a pointwise convex and monotonic forward function.

\subsubsection{The inverse problem with finitely many measurements}

In practical applications, it is natural to assume that the unknown Robin transmission coefficient
is a piecewise constant function, i.e., $\gamma(x)=\sum_{j=1}^n \gamma_j \chi_j(x)$ where 
$\gamma_j>0$ and $\chi_j:=\chi_{\Gamma_j}$ are the characteristic functions
on pairwise disjoint subsets $\Gamma_j\subseteq \Gamma$ of a given partition $\Gamma=\bigcup_{j=1}^n \Gamma_j$.
For the ease of notation, here and in the following, we identify a piecewise constant function $\gamma(x)=\sum_{j=1}^n \gamma_j \chi_j(x)\in L^\infty(\Gamma)$ with
the vector $\gamma=(\gamma_1,\ldots,\gamma_n)\in \R^n$. 
We also simply write $a$ for the constant function $\gamma(x)=a$, and for the vector $(a,\ldots,a)\in \R^n$ (and use $b$ analogously).
Throughout this work we always assume that $n\geq 2$.

It is also natural to assume that one knows bounds $a,b\in \R$ on the unknown Robin transmission coefficient, so that $0<a\leq \gamma_j\leq b$ for all $j=1,\ldots,n$. For the semi-discretized inverse problem of reconstructing the finite-dimensional coefficient vector $\gamma\in [a,b]^n\subset \R^n$ from
the (infinite-dimensional) measurements $\Lambda(\gamma)$, the results in \cite[Thm.~2.1 and 2.2]{harrach2019global} show uniquely solvability and Lipschitz stability. Moroever, \cite[Thm.~5.2]{harrach2019global} shows how to explicitly calculate the Lipschitz constant for a given setting using arguments similar to (and inspiring) section \ref{Sec:Newton} in this work. 

We now go one step further and assume that we can only measure finitely many components of $\Lambda(\gamma)$, i.e.,
that we can measure
\begin{equation}\label{eq:Robin_measurements}
F(\gamma)=\left( \int_{\partial \Omega} g_j \Lambda(\gamma) h_j \dx[s]\right)_{j=1}^m\in \R^m
\end{equation}
for a finite number of Neumann boundary data $g_j,h_j\in L^2(\partial \Omega)$. Hence, the fully discretized 
inverse Robin transmission problem leads to the finite dimensional non-linear inverse problem to
\[
\text{ determine } \quad \gamma\in [a,b]^n\subset \R^n \quad \text{ from } \quad F(\gamma)\in \R^m.
\]

The following practically important questions are then to be answered: Given bounds $[a,b]$ and a partition of $\Gamma$ (i.e., a desired resolution), how many and which Neumann boundary functions $g_j$, $h_j$ 
should be used, so that $F(\gamma)$ uniquely determines $\gamma$? How good is the stability of the resulting inverse problem with regard to noisy measurements?
How can one construct a globally convergent numerical algorithm to practically determine $\gamma$ from $F(\gamma)$?
And how good will the solution be in the case that the true Robin transmission function $\gamma$ is not piecewise constant?

The following subsections show how these questions can be answered using the theory developed in section \ref{Sec:Newton}.
For this, let us first observe, that the symmetric choice $g_j=h_j$ leads to an inverse problem with a pointwise convex and monotonic forward function.

\begin{lemma}\label{lemma:F_rescaled}
Let $b>a>0$, $g_1,\ldots,g_n\in L^2(\partial \Omega)$, and
\[
F:\ (0,\infty)^n\to \R^n, \quad \gamma\mapsto  F(\gamma):=\left( \int_{\partial \Omega} g_j \Lambda(\gamma) g_j \dx[s]\right)_{j=1}^n.
\]
$F$ is a pointwise convex and anti-monotone, continuously differentiable function with locally Lipschitz continuous derivative $F'(\gamma)\in \R^{n\times n}$, where
\[
e_j^T F'(\gamma)e_l= \frac{\partial F_j(\gamma)}{\partial \xi_l}=-\int_{\Gamma_l}  |u_{\gamma}^{g_j}|^2  \dx[s]
\ \text{ for all } \gamma\in (0,\infty)^n,\ j,l\in \{1,\ldots,n\}.
\]

Given a vector $y=(y_j)_{j=1}^n\in \R^n$ with $F(b)\leq y\leq F(a)$, we define the rescaled function
\[
\Phi:\ U:=(-\infty,\textstyle \frac{b}{b-a})^n \subseteq \R^n \to \R^n, 
\quad \displaystyle  \Phi(\xi):=\frac{1}{b-a} \left( F(r(\xi))-y \right).
\]
with $r(\xi):=b\1-(b-a)\xi$. Then $\Phi$ is a pointwise convex and monotonic, continuously differentiable function with locally Lipschitz continuous derivative,
and $\Phi'(\xi)=-F'(r(\xi))$ for all $\xi\in U$. Also, $\Phi(0)\leq 0\leq \Phi(1)$.
\end{lemma}
\begin{proof} 
This follows immediately from lemma~\ref{lemma:Lambda_properties}.
\end{proof}

\subsection{Uniqueness, stability and global Newton convergence}

We summarize the assumptions on the setting: Let $\Omega \subset \R^d$ ($d\geq 2$), be a bounded domain with Lipschitz boundary $\partial \Omega$  and let $D$ be an open subset of $\Omega$, with $\overline D\subset \Omega$, Lipschitz boundary $\Gamma:=\partial D$ and connected complement $\Omega\setminus\overline{D}$.
We assume that the true unknown Robin transmission coefficient $\hat \gamma\in L^\infty_+(\Gamma)$ is bounded by $b\geq \hat \gamma(x)\geq a$ for $x\in \Gamma$ a.e., with a-priori known bounds $b>a>0$, and that $\hat \gamma=\sum_{j=1}^n \hat \gamma_j \chi_j$, $\chi_j:=\chi_{\Gamma_j}$, is piecewise constant on an a-priori known partition $\Gamma=\bigcup_{j=1}^n \Gamma_j$ into $n\in \N$, $n\geq 2$, pairwise disjoint measurable subsets $\Gamma_j\subset \Gamma$ with positive measure.

We will show how to construct $n$ Neumann boundary functions $g_j\in L^2(\partial \Omega)$, so that $\gamma\in [a,b]^n$ is uniquely determined by the $n$ measurements  
\[
F(\gamma)=\left( \int_{\partial \Omega} g_j \Lambda(\gamma) g_j \dx[s]\right)_{j=1}^n\in \R^n.
\] 
We also quantify the Lipschitz stability constant of this inverse problem, and show that the inverse problem can be numerically solved with
a globally convergent Newton method. Our main tool is to reformulate the finite-dimensional inverse problem as a zero finding problem 
for the pointwise monotonic and convex function $\Phi$ introduced in lemma~\ref{lemma:F_rescaled}. 
Then we use a relation to the concept of localized potentials \cite{gebauer2008localized}
to prove that we can choose the measurements $g_j$ in such a way that $\Phi$ also fulfills the additional assumptions on the directional derivative of the forward function as required by our Newton convergence theory in section \ref{Sec:Newton}.

At this point, let us stress that our theory in section \ref{Sec:Newton} allows us to find $m=n$ measurements that uniquely recover the $n$ unknown components of the Robin coefficient, but it also requires to use exactly those $m=n$ measurements to ensure global convergence of Newton's method. In practice, it would be highly desirable to 
utilize additional measurements for the sake of redundancy and error reduction. But our convergence theory 
in section \ref{Sec:Newton} does not cover the case $m>n$ yet, and an extension to, e.g., Gauss-Newton or Levenberg-Marquardt methods seems far from trivial. Moreover, for some applications, it would be desirable to also treat the interior boundary $\Gamma$ as unknown. But it is not clear whether the interplay of parametrization, differentiability and localized potentials results that is required to fulfill the assumptions of section \ref{Sec:Newton}
can also be extended to this case. Hence, in all of the following we will only consider the case where the interior boundary is known and utilize exactly $m=n$ measurements.

\subsubsection{Choosing the measurements (for specific bounds)}\label{subsect:choose_meas_specific}

To demonstrate the key idea, we will first consider the specific (and rather restrictive) case where the bounds $a,b$ fufill
\[
n(n+3)<\frac{b}{b-a},
\] 
since this case can be treated by simply combining Theorem \ref{thm:simple_Newton}
with a known localized potentials result from \cite{harrach2019global}.
The case of general bounds $b>a>0$ is more technical and requires an extended result on simultaneously localized potentials. It will be treated in subsection~\ref{subsect:choose_meas_general}.

\begin{theorem}\label{theorem:choose_meas_specific}
Given bounds $b>a>0$ that fulfill $n(n+3)<\frac{b}{b-a}$, we define the piecewise constant functions
$\kappa^{(j)}\in L^\infty_+(\Gamma)$, $j=1,\ldots,n$, by setting
\[
\kappa^{(j)}:=
\left\{ \begin{array}{l l}
  3b-2a  & \text{ on } \Gamma_j,\\
 b- n(n+3)(b-a)  & \text{ on } \Gamma\setminus \Gamma_j.
\end{array} \right.
\]

\begin{enumerate}[(a)]
\item If, for all $j=1,\ldots,n$, $g_j\in L^2(\partial \Omega)$ fulfills
\begin{equation}\label{eq:lemma_choose_meas_specific_gj}
\lambda^{(j)}:=
\int_{\Gamma_j}  |u_{\kappa^{(j)}}^{g_j}|^2     \dx[s] -
(n^2+3n+1)\int_{\Gamma\setminus \Gamma_j}   |u_{\kappa^{(j)}}^{g_j} |^2   \dx[s]>0,
\end{equation}
then the finite-dimensional non-linear inverse problem of determining 
\[
\gamma=\left(\gamma_j\right)_{j=1}^n\in [a,b]^n \quad \text{ from }\quad
F(\gamma):=\left(\int_{\partial\Omega} g_j \Lambda(\gamma) g_j\dx[s]\right)_{j=1}^n\in \R^n
\]
has a unique solution in $[a,b]^n$, and $\gamma$ depends Lipschitz continuously on $F(\gamma)$.
Moreover, the iterates of Newton's method applied to the problem of determining $\gamma$ from $F(\gamma)$, with 
 initial value $\gamma^{(0)}=a\in \R^n$, quadratically converge to the unique solution $\gamma$ (see lemma~\ref{lemma:F_prop_meas_specific} for 
 more details on the properties of $F$).
\item In any large enough finite-dimensional subspace of $L^2(\partial \Omega)$, one can find
$g_1,\ldots,g_n$ fulfilling \eqref{eq:lemma_choose_meas_specific_gj} by the following construction:

Let $(f_m)_{m\in \N}\subseteq L^2(\partial \Omega)$ be a sequence of vectors with dense linear span in $L^2(\partial \Omega)$. 
Let $j\in \{1,\ldots,n\}$, $m\in \N$, and
let $M^{(j)}_m\in \R^{m\times m}$ be the symmetric matrix with entries ($l,k=1,\ldots,m$)
\[
e_l^T M^{(j)}_m e_k:=\int_{\Gamma_j}  u_{\kappa^{(j)}}^{f_l}  u_{\kappa^{(j)}}^{f_k}    \dx[s] -
(n^2+3n+1)\int_{\Gamma\setminus \Gamma_j}   u_{\kappa^{(j)}}^{f_l}  u_{\kappa^{(j)}}^{f_k}   \dx[s].
\]

Then, for sufficiently large dimension $m\in \N$, the matrix $M^{(j)}_m\in \R^{m\times m}$ has a positive eigenvalue $\lambda^{(j)}>0$.
For a corresponding normalized eigenvector $v^{(j)}=(v^{(j)}_1,\ldots,v^{(j)}_m)\in \R^m$, \eqref{eq:lemma_choose_meas_specific_gj} is fulfilled by
\[
g_j:=\sum_{k=1}^m v^{(j)}_k f_k\in L^2(\partial \Omega). 
\]
\end{enumerate}
\end{theorem}

To prove Theorem~\ref{theorem:choose_meas_specific}, we formulate the consequences
of our Newton convergence theory in section \ref{Sec:Newton} in the following lemma.

\begin{lemma}\label{lemma:F_prop_meas_specific}
If $n(n+3)<\frac{b}{b-a}$ then
\[
\textstyle [a,b]^n\subset I\subset I' \subset (0,\infty)^n.
\]
where $I:=(a-\frac{b-a}{n-1},b+\frac{b-a}{n-1})^n$, and $I':=\left[b-n(b-a),2b-a\right]^n$.
If, for all $j=1,\ldots,n$, $g_j\in L^2(\partial \Omega)$ fulfills \eqref{eq:lemma_choose_meas_specific_gj}, then 
the function
\[
F:\ (0,\infty)^n\to \R^n, \quad F(\gamma):=\left( \int_{\partial \Omega} g_j \Lambda(\gamma) g_j \dx[s]\right)_{j=1}^n\in \R^n
\]
has the following properties 
\begin{enumerate}[(a)]
\item $F$ is pointwise convex, anti-monotone, and continuously differentiable.
\item $F$ is injective on $I'$, and its Jacobian $F'(\gamma)$ is invertible for all $\gamma\in I'$.
\item With $L:=(n+2) \left( \min_{j=1,\ldots,n}\ \lambda^{(j)} \right)^{-1}$, we have that for all $\gamma,\gamma'\in I'$
\[
\norm{\gamma-\gamma'}_\infty \leq L \norm{F(\gamma)-F(\gamma')}_\infty \quad \text{ and } \quad \norm{F'(\gamma)^{-1}}_\infty \leq L.
\]
\item For all $\gamma\in [a,b]^n$, $F(b)\leq F(\gamma) \leq F(a)$. Moreover, for all $y\in \R^n$ with $F(b)\leq  y \leq F(a)$ there exists a unique $\gamma\in I$ with $F(\gamma)=y$. The
Newton iteration 
\[
\gamma^{(i+1)}:=\gamma^{(i)}-F'(\gamma^{(i)})^{-1} \left( F(\gamma^{(i)}) - y \right), \ \text{ started with $\gamma^{(0)}:=a \in \R^n$,}
\]
produces a sequence $(\gamma^{(i)})_{i\in \N}\subset I'$ that converges quadratically to $\gamma$.
\end{enumerate}
\end{lemma}
\begin{proof}
We define $r$ and $\Phi$ as in lemma~\ref{lemma:F_rescaled}.
Then,
\[
\textstyle
r([0,1]^n)=[a,b]^n, \quad 
r((-\frac{1}{n-1}, 1+\frac{1}{n-1} ) ^n)=I, \quad 
r([-1,n]^n)=I',
\]
so that the interval inclusions follow from the anti-monotonicity of $r$.

Lemma~\ref{lemma:F_rescaled} yields assertion (a) and that $\Phi:\ U\subseteq \R^n\to \R^n$ is a continuously differentiable pointwise convex and monotonic function
on $U=(-\infty,\textstyle \frac{b}{b-a})^n$. The assumption $n(n+3)<\frac{b}{b-a}$ guarantees that $U$ contains $[-2,n(n+3)]^n$.
Moreover, using that 
\[
\kappa^{(j)}
=r(-2 e_j + n(n+3) e_j')\in (0,\infty)^n,
\]
it follows from lemma~\ref{lemma:F_rescaled} and \eqref{eq:lemma_choose_meas_specific_gj} that
\begin{align*}
\lefteqn{e_j^T \Phi'(-2 e_j + n(n+3) e_j')\left( e_j  - (n^2+3n+1)  e_j' \right)}\\
&=\int_{\Gamma_j}  |u_{\kappa^{(j)}}^{g_j}|^2  \dx[s] -
(n^2+3n+1) \int_{\Gamma\setminus \Gamma_j}  |u_{\kappa^{(j)}}^{g_j}|^2  \dx[s]
=\lambda^{(j)}>0,
\end{align*}
so that $\Phi$ fulfills the assumptions of Theorem \ref{thm:simple_Newton}.

It follows that $\Phi$ is injective on $[-1,n]^n$, and that for all $\xi,\xi'\in [-1,n]^n$,
$\Phi'(\xi)$ is invertible, 
\[
\norm{\xi-\xi'}_\infty\leq L \norm{\Phi(\xi)-\Phi(\xi')}_\infty, \quad \text{ and } \quad
\norm{\Phi'(\xi)^{-1}}_\infty\leq L.
\]
This proves that $F$ fulfills (b) and (c) on $r([-1,n]^n)=I'$.

The first assertion in (d) follows from the anti-monotonicity of $F$.
For the remaining assertions in (d) note that, by lemma~\ref{lemma:F_rescaled}, $\Phi'$ is also locally Lipschitz continuous, and $\Phi(0)\leq 0\leq \Phi(1)$. Hence, Theorem \ref{thm:simple_Newton}(b) yields
that there exists a unique $\xi\in (-\frac{1}{n-1}, 1+\frac{1}{n-1} ) ^n$ with $\Phi(\xi)=0$, 
and thus a unique $\gamma\in I$ that solves $F(\gamma)=y$.

Moreover, Theorem \ref{thm:simple_Newton}(b) yields that the Newton iteration applied to $\Phi$ with initial value 
$\xi^{(0)}=\1$ does not leave the interval $(-1,n)^n$ and quadratically converges against the unique solution $\xi$ of $\Phi(\xi)=0$.
Since, the Newton iteration is invariant under invertible linear transformations, this yields that the Newton iteration applied to 
$F$ with $\gamma^{(0)}:=a\in \R^n$ does not leave $I'$ and converges quadratically against the unique solution $\gamma$ of $F(\gamma)=y$.
\hfill $\Box$
\kommentar{
Just to be sure:
\begin{align*}
\frac{1}{b-a} \norm{-\gamma +\gamma'}_\infty 
&=\norm{\frac{1}{b-a}(b\1-\gamma) - \frac{1}{b-a}(b\1-\gamma')}_\infty\\
&= \norm{r^{-1}(\gamma)-r^{-1}(\gamma')}_\infty\\ 
&= \norm{\xi-\eta}_\infty \\
&\leq L \norm{\Phi(\xi)-\Phi(\eta)}_\infty\\
&=  L \norm{\frac{1}{b-a}(F(r(\xi))-y) - \frac{1}{b-a}(F(r(\eta))-y)}_\infty\\
&=  L \norm{\frac{1}{b-a} ( F(r(\xi)) - F(r(\eta)) )}_\infty\\
&= \frac{L}{b-a} \norm{F(\gamma)-F(\gamma')},
\end{align*}
}
\end{proof}


\begin{proof}[of Theorem~\ref{theorem:choose_meas_specific}]
\begin{enumerate}[(a)]
\item follows from lemma~\ref{lemma:F_prop_meas_specific}.
\item Let $j\in \{1,\ldots,n\}$. From the localized potentials result in \cite[Lemma 4.3]{harrach2019global}, it follows that there exists $g\in L^2(\partial \Omega)$ with 
\begin{equation*}
\int_{\Gamma_j}  |u_{\kappa^{(j)}}^{g}|^2     \dx[s] -
(n^2+3n+1)\int_{\Gamma\setminus \Gamma_j}   |u_{\kappa^{(j)}}^{g} |^2   \dx[s]>1.
\end{equation*}
By density and continuity, for sufficiently large $m\in \N$, there exists a function
$f\in \mathrm{span}\{f_1,\ldots,f_m \}$ with 
\begin{equation*}
\int_{\Gamma_j}  |u_{\kappa^{(j)}}^{f}|^2     \dx[s] -
(n^2+3n+1)\int_{\Gamma\setminus \Gamma_j}   |u_{\kappa^{(j)}}^{f} |^2   \dx[s]>\frac{1}{2}.
\end{equation*}
Writing $f=\sum_{l=1}^m v_l f_l$ with $v_l\in \R$, and $v:=(v_l)_{l=1}^m\in \R^m$, we thus have
\begin{equation*}
v^T M_m^{(j)} v=\int_{\Gamma_j}  |u_{\kappa^{(j)}}^{f}|^2     \dx[s] -
(n^2+3n+1)\int_{\Gamma\setminus \Gamma_j}   |u_{\kappa^{(j)}}^{f} |^2   \dx[s]>\frac{1}{2}.
\end{equation*}
This shows that the symmetric matrix $M_m^{(j)}\in \R^{m\times m}$ must have a positive eigenvalue when its dimension $m\in \N$ is sufficiently large. On the other hand, every normalized eigenvector $v^{(j)}=(v^{(j)}_1,\ldots,v^{(j)}_m)\in \R^m$ 
corresponding to a positive eigenvector $\lambda^{(j)}>0$ fulfills
\begin{align*}
0 < \lambda^{(j)} &= (v^{(j)})^T M_m^{(j)} v^{(j)}\\
&= \int_{\Gamma_j}  |u_{\kappa^{(j)}}^{g_j}|^2     \dx[s] -
(n^2+3n+1)\int_{\Gamma\setminus \Gamma_j}   |u_{\kappa^{(j)}}^{g_j} |^2   \dx[s],
\end{align*}
with $g_j:=\sum_{k=1}^m v^{(j)}_k f_k\in L^2(\partial \Omega)$, so that (b) is proven.
\hfill $\Box$
\end{enumerate}
\end{proof}

\subsubsection{Choosing the measurements (for general bounds)}\label{subsect:choose_meas_general}

We now show how to treat the case of general bounds $b>a>0$. 


\begin{theorem}\label{theorem:choose_meas_general}
Given $b>a>0$, choose $0<\epsilon<\frac{a}{2(b-a)}$, and set $c:=2+\frac{2}{\epsilon}$, 
$
K:=\operatorname{ceil}\left( \frac{2cn}{\epsilon}\left(1+\epsilon + \frac{1+\epsilon}{cn}\right)\right),
$
and $\beta:= \frac{1+\epsilon+cn+2\epsilon cn}{\epsilon}$.
Let $\kappa^{(j,k)}\in L^\infty_+(\Gamma)$ denote the piecewise constant function
\begin{align}\label{eq:def_kappa_jk}
\kappa^{(j,k)}:=
&\left\{ \begin{array}{l l}
  b- (b-a)\left( -\frac{1+\epsilon}{cn}+(k-2)\frac{\epsilon}{2cn} \right)  & \text{ on } \Gamma_j,\\
  a-2\epsilon(b-a) & \text{ on } \Gamma\setminus \Gamma_j. 
\end{array} \right.\\
\end{align}

\begin{enumerate}[(a)]
\item If $g_j\in L^2(\partial \Omega)$, $j=1,\ldots,n$, fulfills 
\begin{align}\label{eq:lemma_choose_meas_general_gj}
\lambda_{j,k}&:=
\frac{1}{2}\int_{\Gamma_j}  |u_{\kappa^{(j,k)}}^{g_j}|^2  \dx[s]
- \beta \int_{\Gamma\setminus \Gamma_j} |u_{\kappa^{(j,k)}}^{g_j}|^2  \dx[s]>0,
\end{align}
for all $k=1,\ldots,K$, then the finite-dimensional non-linear inverse problem of determining 
\[
\gamma=\left(\gamma_j\right)_{j=1}^n\in [a,b]^n \quad \text{ from }\quad
F(\gamma):=\left(\int_{\partial\Omega} g_j \Lambda(\gamma) g_j\dx[s]\right)_{j=1}^n\in \R^n
\]
has a unique solution in $[a,b]^n$, and $\gamma$ depends Lipschitz continuously $\Lambda(\gamma)$.
Moreover, the iterates of Newton's method applied to the problem of determining $\gamma$ from $F(\gamma)$, with 
initial value $\gamma^{(0)}=a\in \R^n$, quadratically converge to the unique solution $\gamma$ (see lemma~\ref{lemma:F_prop_meas_general} for 
more details on the properties of $F$).
\item In any large enough finite dimensional subspace of $L^2(\partial \Omega)$, one can find $g_1,\ldots,g_n$ fulfilling
\eqref{eq:lemma_choose_meas_general_gj} by the following construction: 

Let $(f_m)_{m\in \N}\subseteq L^2(\partial \Omega)$ be a sequence of vectors with dense linear span in $L^2(\partial \Omega)$. 
Let $j\in \{1,\ldots,n\}$, $m\in \N$, and $v^{(j)}=(v^{(j)}_1,\ldots,v^{(j)}_m)\in \R^m$
be a normalized eigenvector corresponding to a largest eigenvalue of the symmetric matrix
\[
M_m^{(j)}:=\frac{\alpha}{2} S^{(j)} - \beta \sum_{k=1}^K  R^{(j,k)} \in \R^{m\times m},
\]
where $\alpha:=( 1+ \frac{b-a}{b} (1+\epsilon + \frac{1+\epsilon}{cn}))^{-1}$, and for $i,l=1,\ldots,m$,
\[
e_i^T S^{(j)} e_l:=\int_{\Gamma}   u_{\kappa^{(j,1)}}^{f_i}  u_{\kappa^{(j,1)}}^{f_l}   \dx[s],
\ \text{ and } \
e_i^T R^{(j,k)} e_l:=\int_{\Gamma\setminus \Gamma_j}   u_{\kappa^{(j,k)}}^{f_i}  u_{\kappa^{(j,k)}}^{f_l}   \dx[s],
\]
Then, for sufficiently large dimension $m\in \N$, \eqref{eq:lemma_choose_meas_general_gj} is fulfilled by
\[
g_{j}:=\sum_{l=1}^m v^{(j)}_l f_l\in \mathrm{span}\{f_1,\ldots,f_m \}\subset L^2(\partial \Omega).
\]
\end{enumerate}
\end{theorem}

As in subsection \ref{subsect:choose_meas_specific}, we prove Theorem~\ref{theorem:choose_meas_general}(a) by applying
our Newton convergence theory from section \ref{Sec:Newton} in the following lemma.

\begin{lemma}\label{lemma:F_prop_meas_general}
We have that
\[
\textstyle [a,b]^n\subseteq I:=\left[ a-(b-a)\epsilon, b+(b-a)\frac{1+\epsilon}{cn}\right]^n
\subseteq (0,\infty)^n.
\]
If, for all $j=1,\ldots,n$, $g_j\in L^2(\partial \Omega)$ fulfills \eqref{eq:lemma_choose_meas_general_gj} for all $k=1,\ldots,K$, then 
the function
\[
F:\ (0,\infty)^n\to \R^n, \quad F(\gamma):=\left( \int_{\partial \Omega} g_j \Lambda(\gamma) g_j \dx[s]\right)_{j=1}^n\in \R^n
\]
has the following properties:
\begin{enumerate}[(a)]
\item $F$ is pointwise convex, anti-monotone, and continuously differentiable.
\item $F$ is injective on $I$, and its Jacobian $F'(\gamma)$ is invertible for all $\gamma\in I$.
\item For all $\gamma,\gamma'\in I$,
\[
\norm{\gamma-\gamma'}_\infty \leq L \norm{F(\gamma)-F(\gamma')}_\infty \quad \text{ and } \quad \norm{F'(\gamma)^{-1}}_\infty \leq L,
\]
where $L:=\left( \min_{j=1,\ldots,n,\ k=1,\ldots,K}\ \lambda_{j,k} \right)^{-1}$.
\item For all $\gamma\in [a,b]^n$, $F(b)\leq F(\gamma) \leq F(a)$. Moreover, for all $y\in \R^n$ with $F(b)\leq y \leq F(a)$ there exists a unique $\gamma\in I$ with $F(\gamma)=y$. The
Newton iteration 
\[
\gamma^{(i+1)}:=\gamma^{(i)}-F'(\gamma^{(i)})^{-1} \left( F(\gamma^{(i)}) - y \right), \ \text{ started with } \
\gamma^{(0)}:=a \in \R^n,
\]
produces a sequence $(\gamma^{(i)})_{i\in \N}\subset I$ that converges quadratically to $\gamma$.
\end{enumerate}
\end{lemma}
\begin{proof}
We define $r$ and $\Phi$ as in lemma~\ref{lemma:F_rescaled}.
Then $r([0,1]^n)=[a,b]^n$ and 
\[
\textstyle
r([-\frac{1+\epsilon}{cn},1+\epsilon]^n)=[a-(b-a)\epsilon, b+(b-a)\frac{1+\epsilon}{cn}]^n.
\]

Lemma~\ref{lemma:F_rescaled} yields assertion (a) and that $\Phi:\ U\subseteq \R^n\to \R^n$ is a continuously differentiable pointwise convex and monotonic function
on $U=(-\infty,\textstyle \frac{b}{b-a})^n$, with locally Lipschitz continuous $\Phi'$, and $\Phi(0)\leq 0\leq \Phi(1)$.
Also, $U$ contains $[-\frac{1+\epsilon}{cn}-\frac{\epsilon}{2cn},1+2\epsilon]^n$ since $\epsilon<\frac{a}{2(b-a)}$.

Moreover, with $z^{(j,k)},d^{(j)}\in \R^n$ defined by \eqref{eq:Newton_tight_def_z} and \eqref{eq:Newton_tight_def_d},
we have that 
\[
\kappa^{(j,k)}=r(z^{(j,k)}) \quad \text{ and } \quad d^{(j)}=\frac{1}{2}e_j-\beta e_j'.
\]
It thus
follows from lemma~\ref{lemma:F_rescaled} and \eqref{eq:lemma_choose_meas_general_gj} that
\begin{align*}
e_j^T \Phi'(z^{(j,k)})d^{(j)}
=\frac{1}{2}\int_{\Gamma_j}  |u_{\kappa^{(j,k)}}^{g_j}|^2  \dx[s]
- \beta \int_{\Gamma\setminus \Gamma_j} |u_{\kappa^{(j,k)}}^{g_j}|^2  \dx[s]=\lambda_{j,k}>0,
\end{align*}
so that $\Phi$ fulfills the assumptions of Theorem~\ref{thm:tight_Newton} which then yields the assertions (b)--(d).
\hfill $\Box$
\end{proof}

To prove Theorem~\ref{theorem:choose_meas_general}(b), we need to ensure that 
there exist Neumann data $g_j\in L^2(\partial \Omega)$ so that the corresponding solutions 
$u_{\kappa^{(j,k)}}^{g_j}$ are much larger on $\Gamma_j$ than on $\Gamma\setminus \Gamma_j$,
and this property has to hold for several Robin transmission coefficients $\kappa^{(j,k)}$, $k=1,\ldots,K$, simultaneously.

Note that for fixed $j\in \{1,\ldots,n\}$ the Robin transmission coefficients $\kappa^{(j,k)}$ only differ on $\Gamma_j$.
Hence, the following lemma will allow us to estimate $u_{\kappa^{(j,k)}}^{g_j}$ on $\Gamma_j$
by $u_{\kappa^{(j,1)}}^{g_j}$.

\begin{lemma}\label{lemma:norm_control}
Let $\gamma^{(1)},\gamma^{(2)}\in L^\infty_+(\Gamma)$ with $\gamma^{(1)}=\gamma^{(2)}$ on $\Gamma\setminus \Gamma_0$,
where $\Gamma_0$ is a measurable subset of $\Gamma$ with positive measure. Then, for all $g\in L^2(\partial \Omega)$, the corresponding solutions
$u_1,u_2\in H^1(\Omega)$ of \eqref{eq:Robin1}--\eqref{eq:Robin4} with $\gamma=\gamma^{(1)}$, and $\gamma=\gamma^{(2)}$, respectively, fulfill
\begin{align*}
\norm{u_2}_{L^2(\Gamma_0)}
\leq \left( 1+\frac{\norm{\gamma^{(1)}-\gamma^{(2)}}_{L^\infty(\Gamma_0)}}{\inf (\gamma^{(2)}|_{\Gamma_0})}  \right) \norm{u_1}_{L^2(\Gamma_0)}.
\end{align*} 
\end{lemma}
\begin{proof}
We proceed analogously to \cite[Lemma~3.6]{harrach2019fractional_I}. It follows from the variational formulation \eqref{Robin_variational} that 
\begin{align*}
\lefteqn{\inf (\gamma^{(2)}|_{\Gamma_0}) \norm{u_2-u_1}_{L^2(\Gamma_0)}^2}\\
&\leq \int_\Omega\sigma |\nabla (u_2-u_1)|^2 \dx+\int_\Gamma \gamma^{(2)} |u_2-u_1|^2 \dx[s]\\
&=\int_{\partial\Omega}g(u_2-u_1)\dx[s]  - \int_\Omega\sigma \nabla u_1 \cdot \nabla (u_2-u_1) \dx-\int_\Gamma \gamma^{(2)} u_1 (u_2-u_1) \dx[s]\\
&= \int_\Gamma (\gamma^{(1)}-\gamma^{(2)}) u_1 (u_2-u_1) \dx[s]= \int_{\Gamma_0} (\gamma^{(1)}-\gamma^{(2)}) u_1 (u_2-u_1) \dx[s]\\
&\leq \norm{\gamma^{(1)}-\gamma^{(2)}}_{L^\infty(\Gamma_0)} \norm{u_1}_{L^2(\Gamma_0)}  \norm{u_2-u_1}_{L^2(\Gamma_0)}.
\end{align*}
Hence, we obtain
\begin{align*}
\norm{u_2}_{L^2(\Gamma_0)}-\norm{u_1}_{L^2(\Gamma_0)}
&\leq \norm{u_2-u_1}_{L^2(\Gamma_0)}\\
&\leq \frac{\norm{\gamma^{(1)}-\gamma^{(2)}}_{L^\infty(\Gamma_0)}}{\inf (\gamma^{(2)}|_{\Gamma_0})}  \norm{u_1}_{L^2(\Gamma_0)},
\end{align*}
and the assertion follows.\hfill $\Box$
\end{proof}

The next lemma will allow us to construct $u_{\kappa^{(j,k)}}^{g_j}$ for which $u_{\kappa^{(j,1)}}^{g_j}$ is large on $\Gamma_j$ (and thus by lemma~\ref{lemma:sim_loc_pot} all $u_{\kappa^{(j,k)}}^{g_j}$ are large on $\Gamma_j$), and at the same time
all $u_{\kappa^{(j,k)}}^{g_j}$ are small on $\Gamma\setminus \Gamma_j$.

\begin{lemma}\label{lemma:sim_loc_pot}
Let $\Gamma_0$ be a measurable subset of $\Gamma$ with positive measure, $K\in \N$, and $\gamma^{(1)},\gamma^{(2)},\ldots,\gamma^{(K)}\in L^\infty_+(\Gamma)$.

Then, for all $C>0$, there exists $g\in L^2(\partial \Omega)$, so that the corresponding solutions 
$u_1,u_2,\ldots,u_K\in H^1(\Omega)$ of \eqref{eq:Robin1}--\eqref{eq:Robin4} fulfill
\[
\int_{\Gamma_0} |u_1|^2 \dx[s]\geq C \quad \text{ and } \quad \sum_{k=1}^K \int_{\Gamma\setminus \Gamma_0} |u_k|^2 \dx[s]\leq \frac{1}{C}.
\]
\end{lemma}
\begin{proof}
The existence of simultaneously localized potentials for the fractional Schr\"o\-ding\-er equation has recently been shown
in \cite[Theorem 3.11]{harrach2020monotonicity}, and we proceed similarly in this proof. Following the original 
localized potentials approach in \cite{gebauer2008localized}, we start by reformulating the assertion as operator range
\text{(non-)inclusions}, by introducing the operators
\begin{align*}
A_0:\ L^2(\Gamma_0)\to L^2(\partial \Omega), \quad f\mapsto Af:=v_0|_{\partial \Omega},\\
A_k:\ L^2(\Gamma\setminus \Gamma_0)\to L^2(\partial \Omega), \quad f\mapsto Af:=v_k|_{\partial \Omega},
\end{align*}
where $k\in \{1,\ldots,K\}$, $v_0\in H^1(\Omega)$ solves
\begin{equation}\label{eq:lemma:sim_pot_v_0}
\int_\Omega \sigma \nabla v_0\cdot \nabla w \dx + \int_{\Gamma} \gamma^{(1)} v_0 w\dx[s]
= \int_{\Gamma_0} f w \dx[s] \quad \text{ for all } w\in H^1(\Omega),
\end{equation}
and $v_k\in H^1(\Omega)$ solves
\begin{equation}\label{eq:lemma:sim_pot_v_k}
\int_\Omega \sigma \nabla v_k\cdot \nabla w \dx + \int_{\Gamma} \gamma^{(k)} v_k w\dx[s]
= \int_{\Gamma\setminus \Gamma_0} f w \dx[s] \quad \text{ for all } w\in H^1(\Omega).
\end{equation}
It is easily shown (see, e.g., the proof of \cite[Theorem 3.1]{harrach2019global}) that the adjoints of these operators 
are given by
\begin{alignat*}{2}
A_0^*:&\ L^2(\partial \Omega)\to L^2(\Gamma_0), & \quad g&\mapsto u_1|_{\Gamma_0},\\
A_k^*:&\ L^2(\partial \Omega)\to L^2(\Gamma\setminus \Gamma_0),& \quad g&\mapsto u_k|_{\Gamma\setminus \Gamma_0},
\end{alignat*}
where $u_1,\ldots,u_K\in H^1(\Omega)$ solve \eqref{eq:Robin1}--\eqref{eq:Robin4} with Neumann boundary data $g$ and
Robin transmission coefficients $\gamma^{(1)},\ldots,\gamma^{(K)}$, respectively.


By a simple normalization argument, the assertion is now equivalent to showing that 
\begin{equation}\label{eq:lemma:sim_pot_notexists}
\not\exists C>0:\ 
\norm{A_0^* g}^2 \leq C \sum_{k=1}^K \norm{A_k^* g}^2 =C \left\| \begin{pmatrix} A_1^*\\ \vdots \\ A_K^* \end{pmatrix} g \right\|^2 \quad \forall g\in L^2(\partial \Omega).
\end{equation}
Using a functional analytic relation between operator ranges and the norms or their adjoints (cf., \cite[Lemma 2.5]{gebauer2008localized}, \cite[Cor.~3.5]{fruhauf2007detecting}), the property \eqref{eq:lemma:sim_pot_notexists} (and thus the assertion) is proven if we can show that
\begin{equation}\label{eq:range_non_inclusion}
\range (A_0)\not\subseteq \range \begin{pmatrix} A_1 & \ldots & A_K \end{pmatrix}= \range (A_1)+\ldots+\range(A_K). 
\end{equation}
We prove \eqref{eq:range_non_inclusion} by contradiction, and assume that 
\[
\range (A_0)\subseteq \range (A_1)+\ldots+\range(A_K). 
\]
Then, for every $f_0\in L^2(\Gamma_0)$, there exist $f_1,\ldots,f_K\in L^2(\Gamma\setminus \Gamma_0)$, so that
\[
A_0 f_0= A_1 f_1+\ldots + A_K f_K.
\]
Let $v_0,\ldots,v_K\in H^1(\Omega)$ be the associated solutions from the definition of $A_0,\ldots,A_K$ (with $f=f_k$), and 
set $v:=v_1+\ldots+v_K-v_0$. Then $v|_{\partial \Omega}=0$, and $\partial_\nu v|_{\partial \Omega}=0$, so that by unique continuation 
$v=0$ in $\Omega\setminus \overline D$. But this also yields that $v|_{\Gamma}=0$, and from this we obtain 
that $v=0$ in $D$, so that $v=0$ in all of $\Omega$.

Hence, using \eqref{eq:lemma:sim_pot_v_0} and \eqref{eq:lemma:sim_pot_v_k}, we obtain for all $w\in H^1(\Omega)$,
\begin{align*}
0&=\int_\Omega \sigma \nabla v\cdot \nabla w \dx + \int_{\Gamma} \gamma^{(1)} v w\dx[s]\\
&=\sum_{k=1}^K \left( \int_\Omega \sigma \nabla v_k\cdot \nabla w \dx + \int_{\Gamma} \gamma^{(k)} v_k w\dx[s]
+ \int_{\Gamma} (\gamma^{(1)}-\gamma^{(k)}) v_k w\dx[s]\right)\\
&\qquad {} - \int_\Omega \sigma \nabla v_0\cdot \nabla w \dx - \int_{\Gamma} \gamma^{(1)} v_0 w\dx[s]\\
&=\sum_{k=1}^K \left( \int_{\Gamma\setminus\Gamma_0} f_k w\dx[s] + \int_{\Gamma} (\gamma^{(1)}-\gamma^{(k)}) v_k w\dx[s]\right)
- \int_{\Gamma_0} f_0 w\dx[s],
\end{align*}
and this shows that 
\[
f_0=\sum_{k=2}^K (\gamma^{(1)}-\gamma^{(k)}) v_k|_{\Gamma_0}.
\]
However, since this holds for all $f_0\in L^2(\Gamma_0)$, 
this would imply that
\[
L^2(\Gamma_0)=\range( \mathcal{M}_2 \mathrm{tr}_{\Gamma_0} )+\ldots+\range( \mathcal{M}_K \mathrm{tr}_{\Gamma_0} )
= \range \left( \begin{pmatrix} \mathcal{M}_2 \mathrm{tr}_{\Gamma_0} & \ldots & \mathcal{M}_K \mathrm{tr}_{\Gamma_0} \end{pmatrix}\right) 
\]
where 
\[
\mathrm{tr}_{\Gamma_0}:\ H^1(\Omega)\to L^2(\Gamma_0), \quad \text{ and } \quad
\mathcal{M}_k:\ L^2(\Gamma_0)\to L^2(\Gamma_0)
\]
are the compact trace operator and the continuous multiplication operator by $\gamma^{(1)}-\gamma^{(k)}$. Hence, the closed infinite-dimensional space $L^2(\Gamma_0)$ would be
the range of a compact operator, which is not possible cf., e.g., \cite[Thm.~4.18]{rudin1991functional}.
This contradiction shows that \eqref{eq:range_non_inclusion} must hold, and thus the assertion is proven.
\hfill $\Box$
\end{proof}

\begin{proof}[of Theorem \ref{theorem:choose_meas_general}]
\begin{enumerate}[(a)]
\item follows from lemma~\ref{lemma:F_prop_meas_general}.
\item Let $j\in \{1,\ldots,n\}$.
As in the proof of Theorem \ref{theorem:choose_meas_specific}(b), it follows from the simultaneously localized potentials result in 
lemma~\ref{lemma:sim_loc_pot}, and a density argument, that $M_m^{(j)}\in \R^{m\times m}$ 
will have a positive eigenvalue if the dimension $m\in \N$ is sufficiently large. Hence, 
for sufficiently large $m$, an eigenvector $v^{(j)}\in \R^m$ corresponding to a largest eigenvalue of $M_m^{(j)}$
will fulfill (with $g_j:=\sum_{l=1}^m v^{(j)}_l f_l\in L^2(\partial \Omega)$)
\begin{align*}
0 \leq (v^{(j)})^T M_m^{(j)} v^{(j)}
& =\frac{\alpha}{2}  \int_{\Gamma_j}  |u_{\kappa^{(j,1)}}^{g_j}|^2  \dx[s]
 - \beta \sum_{k=1}^K  \int_{\Gamma\setminus \Gamma_j} |u_{\kappa^{(j,k)}}^{g_j}|^2  \dx[s]. 
\end{align*}

Using that for all $k\in \{1,\ldots,K\}$
\begin{align*}
\inf (\kappa^{(j,1)}|_{\Gamma_j})&= 
b+ (b-a)\left( \frac{2+3\epsilon}{2cn}\right)\geq b,\\
\norm{\kappa^{(j,k)} - \kappa^{(j,1)}}_{L^\infty(\Gamma_j)},
&\leq (b-a) (K-1)\frac{\epsilon}{2cn} \leq (b-a) \left(1+\epsilon + \frac{1+\epsilon}{cn}\right),
\end{align*}
we have that
\[
 1+ \frac{\norm{\kappa^{(j,k)} - \kappa^{(j,1)}}_{L^\infty(\Gamma_j)}}{\inf (\kappa^{(j,1)}|_{\Gamma_j})} 
 \leq 1+ \frac{b-a}{b} \left(1+\epsilon + \frac{1+\epsilon}{cn}\right)
=\frac{1}{\alpha},
\]
and thus it follows from lemma~\ref{lemma:norm_control}
\begin{align*}
\lefteqn{\frac{1}{2} \int_{\Gamma_j}  |u_{\kappa^{(j,k)}}^{g_j}|^2  \dx[s]
- \beta \int_{\Gamma\setminus \Gamma_j} |u_{\kappa^{(j,k)}}^{g_j}|^2  \dx[s]}\\
& \geq \frac{\alpha}{2}  \int_{\Gamma_j}  |u_{\kappa^{(j,1)}}^{g_j}|^2  \dx[s] - \beta  \sum_{k'=1}^K  \int_{\Gamma\setminus \Gamma_j} |u_{\kappa^{(j,k')}}^{g_j}|^2  \dx[s] \geq 0.
\end{align*}
Hence, \eqref{eq:lemma_choose_meas_general_gj} is fulfilled. 
\hfill $\Box$
\end{enumerate}
\end{proof}

\begin{remark}
Regarding the formulation of Theorem \ref{theorem:choose_meas_general}(b), note that we actually prove that the matrix
$M_m^{(j)}\in \R^{m\times m}$ has a positive eigenvalue if the dimension $m$ is sufficiently large, and that an eigenvector
corresponding to a positive eigenvalue leads to a boundary current $g_j$ that fulfills \eqref{eq:lemma_choose_meas_general_gj}.
But the estimates that we use in the proof of \ref{theorem:choose_meas_general}(b) are far from being sharp, so that it seems worth
checking \eqref{eq:lemma_choose_meas_general_gj} already for eigenvectors to a largest eigenvalue that is not yet positive.
\end{remark}


\subsection{Numerical results}\label{subsect:numerics}

We test our results on the simple example setting shown in figure \ref{fig:Setting}. $\Omega\subset \R^2$ is the two-dimensional unit circle,
and $D$ is a square with corner coordinates  $(0,-0.1)$, $(0,0.3)$, $(-0.4,0.3)$ and $(-0.4,-0.1)$. 
The boundary $\Gamma=\partial D$ is decomposed into $n=4$ parts $\Gamma_1,\ldots,\Gamma_4$ denoting (in this order) the
right, top, left, and bottom side of the square. We assume that the unknown true Robin transmission coefficient $\hat \gamma:\ \Gamma\to \R$ is a-priori known to be bounded by $a:=1$ and $b:=2$ and that it is known to be piecewise-constant with respect to the partition of $\Gamma$, i.e., 
\[
\hat \gamma(x)=\hat \gamma_1 \chi_{\Gamma_1} + \hat \gamma_2 \chi_{\Gamma_2} + \hat \gamma_3 \chi_{\Gamma_3} + \hat \gamma_4 \chi_{\Gamma_4}\quad \text{ with } \quad \hat \gamma_1,\ldots,\hat\gamma_4\in [a,b]=[1,2].
\]
Recall that, for the ease of notation, we  identify a piecewise-constant function $\gamma\in L^2(\Gamma)$ with the vector $(\gamma_1,\ldots,\gamma_4)\in \R^4$,
and simply write $a$ for the constant function $\gamma(x)=a$, and for the vector $(a,a,a,a)\in \R^4$ (and use $b$ analogously).

\begin{figure}
\begin{center}
\mbox{\includegraphics[height=0.25\textheight]{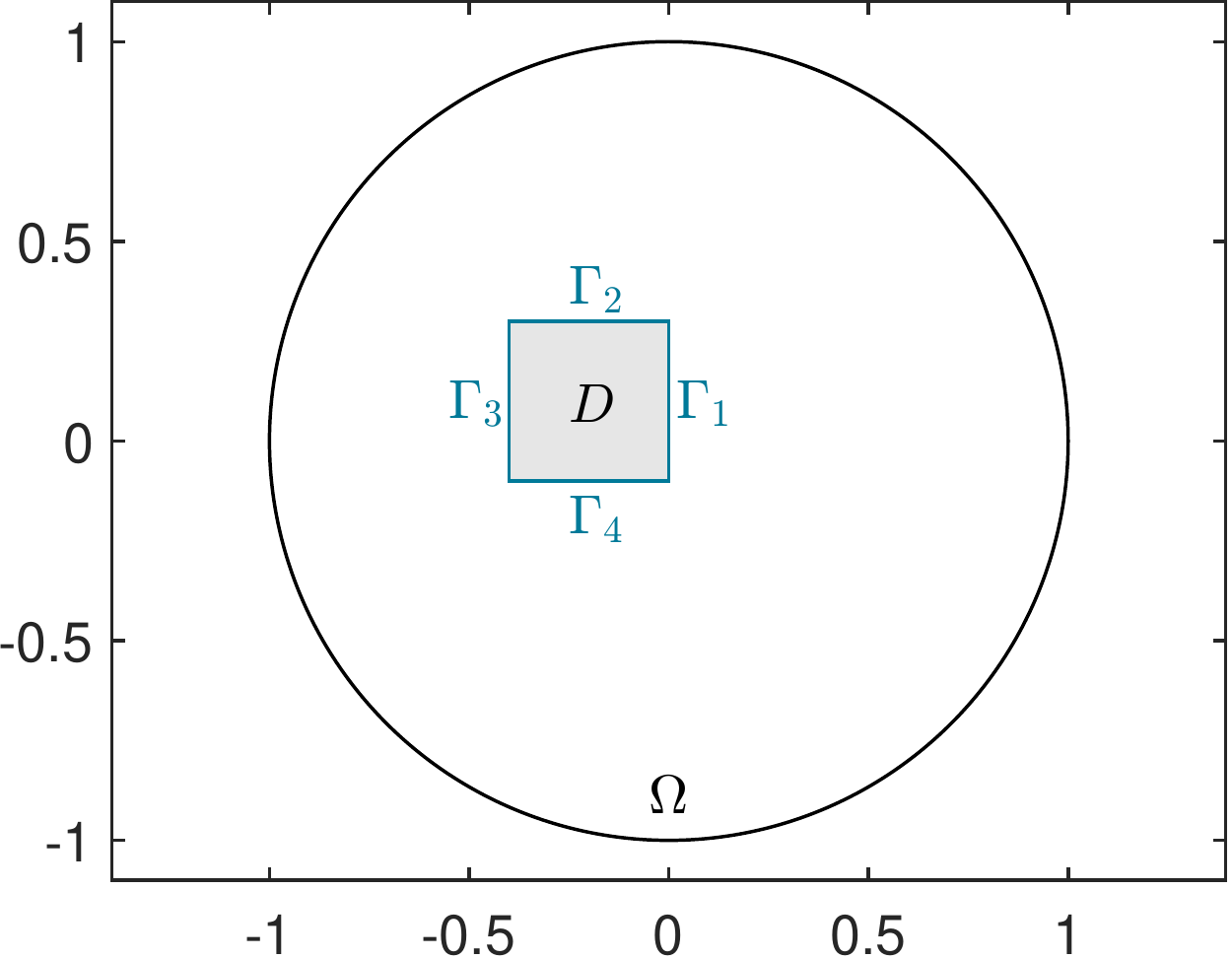}} 
\end{center}
\caption{Setting considered for the numerical results.}
\label{fig:Setting}
\end{figure}

\subsubsection{Choosing the measurements}\label{subsect:measurements}

We first apply Theorem \ref{theorem:choose_meas_general} to construct $n=4$ Neumann boundary functions $g_1,\ldots, g_4\in L^2(\partial \Omega)$ so that
the four measurements 
\[
F(\gamma):=\left(\int_{\partial \Omega} g_j \Lambda(\gamma) g_j \dx[s]\right)_{j=1}^4\in \R^4
\text{ uniquely determine } \gamma\in [a,b]^4\subset \R^4,
\]
and the Newton method applied to $F$ globally converges.

To implement Theorem \ref{theorem:choose_meas_general},
we choose $\epsilon:=0.9\,\frac{a}{2(b-a)}$, which yields $K=173$, 
and $(f_m)_{m\in \N}\subset L^2(\partial \Omega)$ as the standard trigonometric polynomial basis functions 
\[
(1,\cos(\varphi),\sin(\varphi),\cos(2\varphi),\sin(2\varphi),\ldots)\subseteq L^2(\partial \Omega)
\]
with $\varphi$ denoting the angle of a point on the unit circle $\partial \Omega$. 
For each $j\in \{1,\ldots,n\}$, we then calculate the matrix $M_m^{(j)}\in \R^{m\times m}$ starting
with $m=1$, and increase the dimension $m\in \N$, until an eigenvector $v^{(j)}\in \R^m$ corresponding to a largest eigenvalue of this matrix has the property that 
\[
g_j:=\sum_{l=1}^m v^{(j)}_l f_l = v^{(j)}_1 + v^{(j)}_2  \cos(\varphi)+ v^{(j)}_3 \sin(\varphi) + v^{(j)}_4 \cos(2\varphi)+\ldots\in L^2(\partial \Omega)
\]
fulfills \eqref{eq:lemma_choose_meas_general_gj} for all $k=1,\ldots,K$.
To calculate the entries of $M_m^{(j)}$, and for checking \eqref{eq:lemma_choose_meas_general_gj}, 
the required solutions $u^{f_i}_{\kappa^{(j,k)}}$ of the Robin transmission problem \eqref{eq:Robin1}--\eqref{eq:Robin4} with Neumann boundary functions $f_i$, $i=1,\ldots,m$ and Robin transmission coefficients $\kappa^{(j,k)}$, as defined in \eqref{eq:def_kappa_jk}, were obtained using the commercial finite element software Comsol.

For our setting we had to increase the dimension up to at most $m=15$, so that all $g_j$
are trigonometric polynomials of order less or equal $7$. Figure \ref{fig:boundary_currents} shows the boundary functions $g_j$ plotted with respect to the boundary angle on the unit circle $\partial \Omega$.

\begin{figure}
\begin{center}
\begin{tabular}{c@{\ } c@{\ } c@{\ } c}
\mbox{\includegraphics[width=0.23\textwidth]{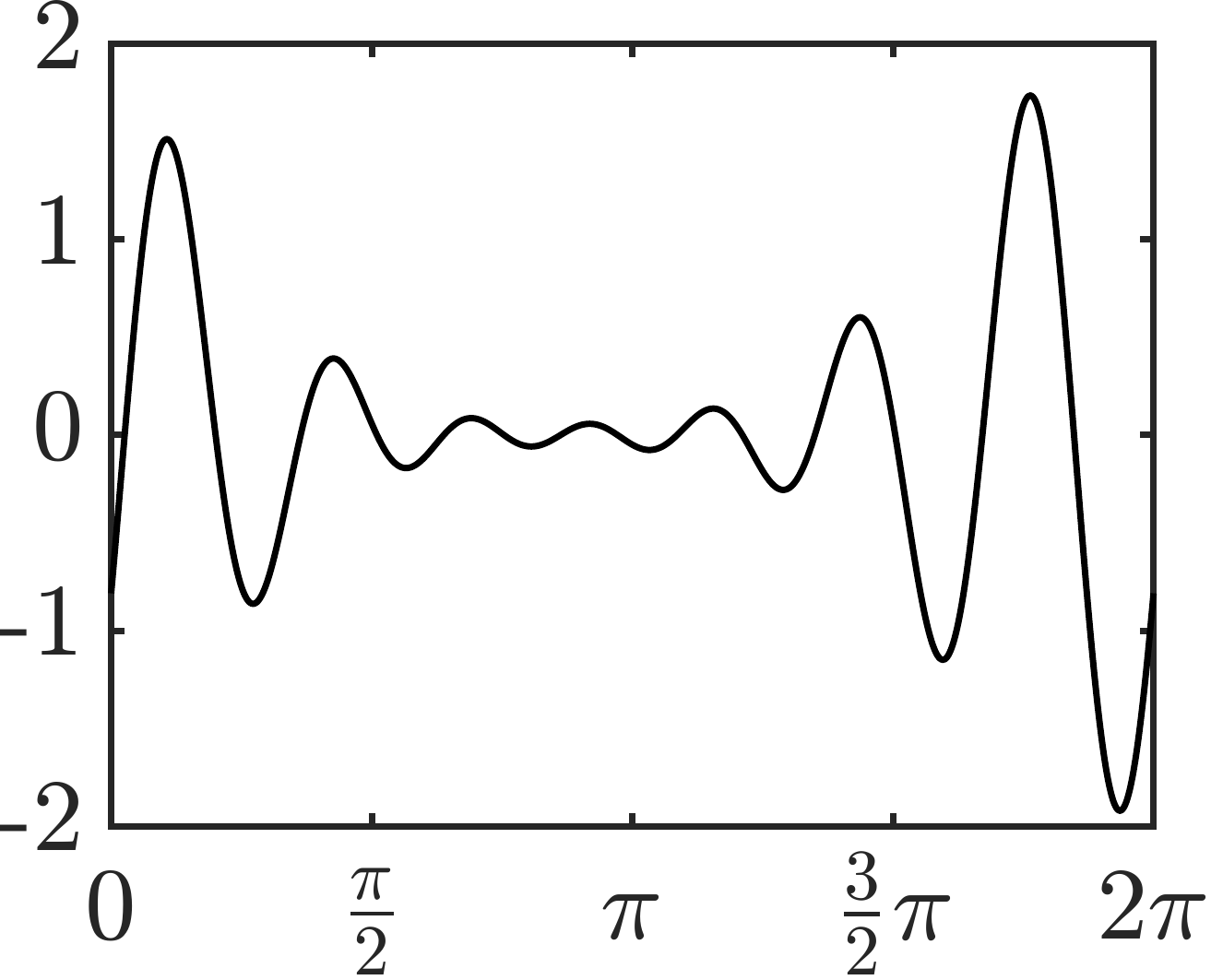}} &
\mbox{\includegraphics[width=0.23\textwidth]{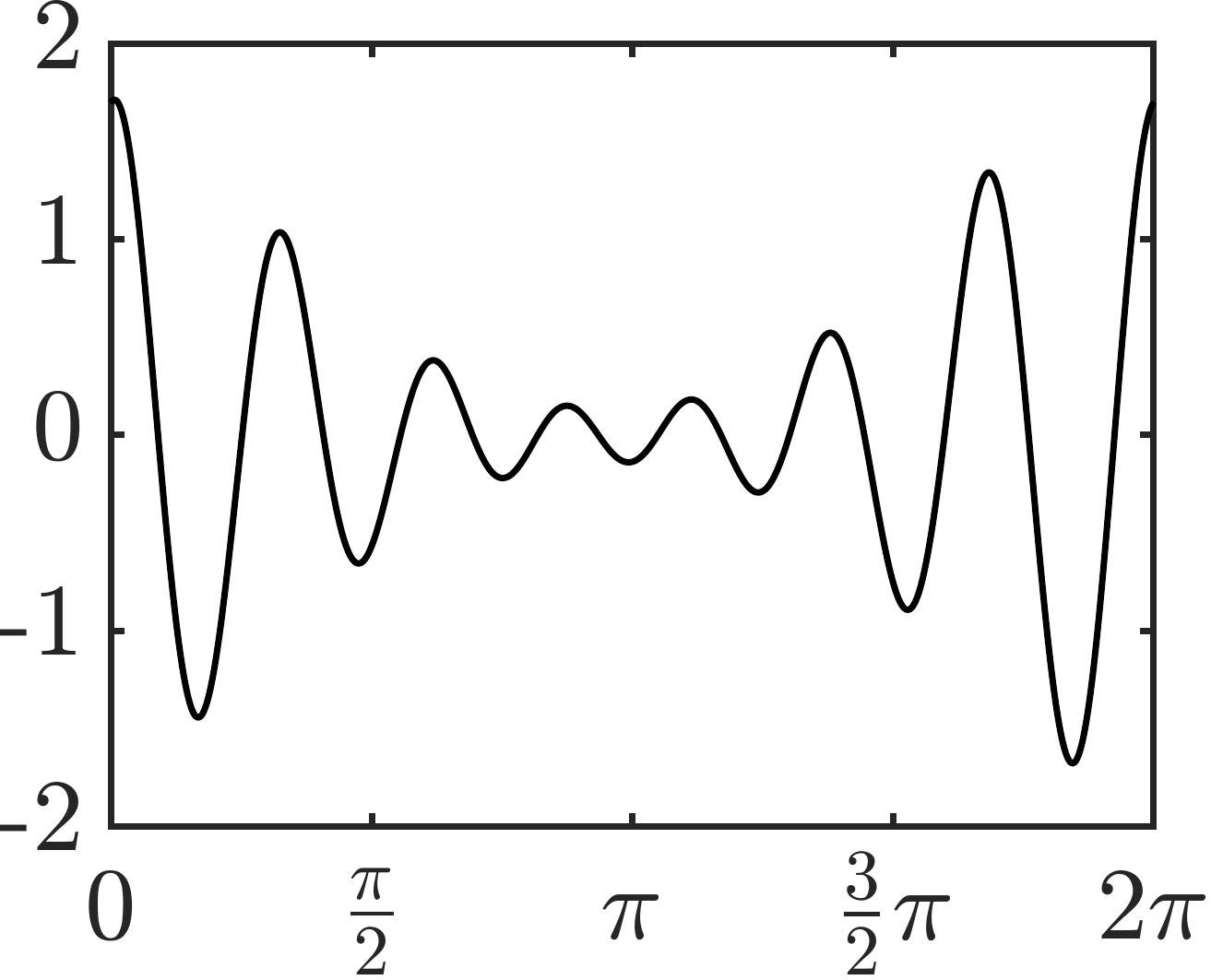}} &
\mbox{\includegraphics[width=0.23\textwidth]{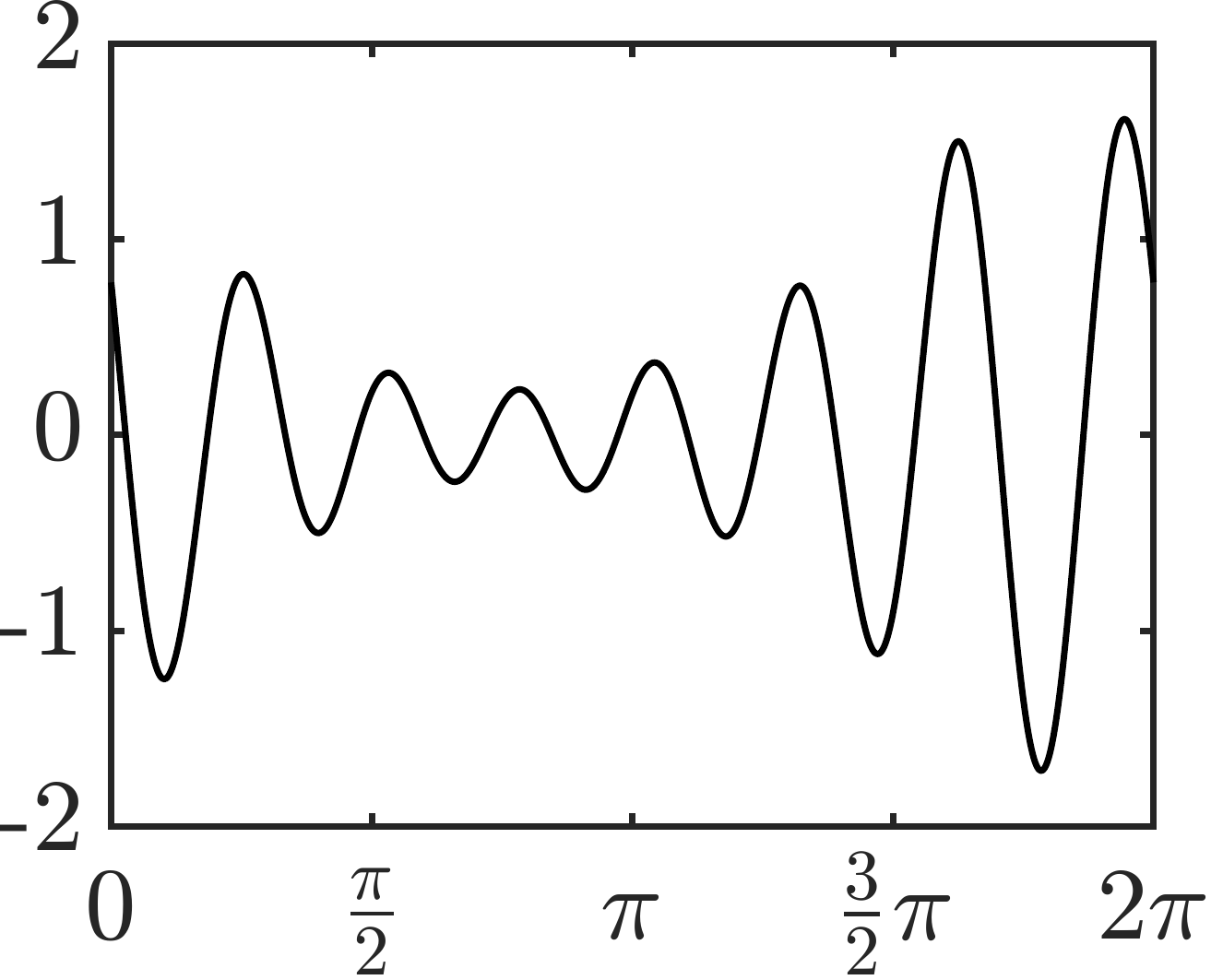}} &
\mbox{\includegraphics[width=0.23\textwidth]{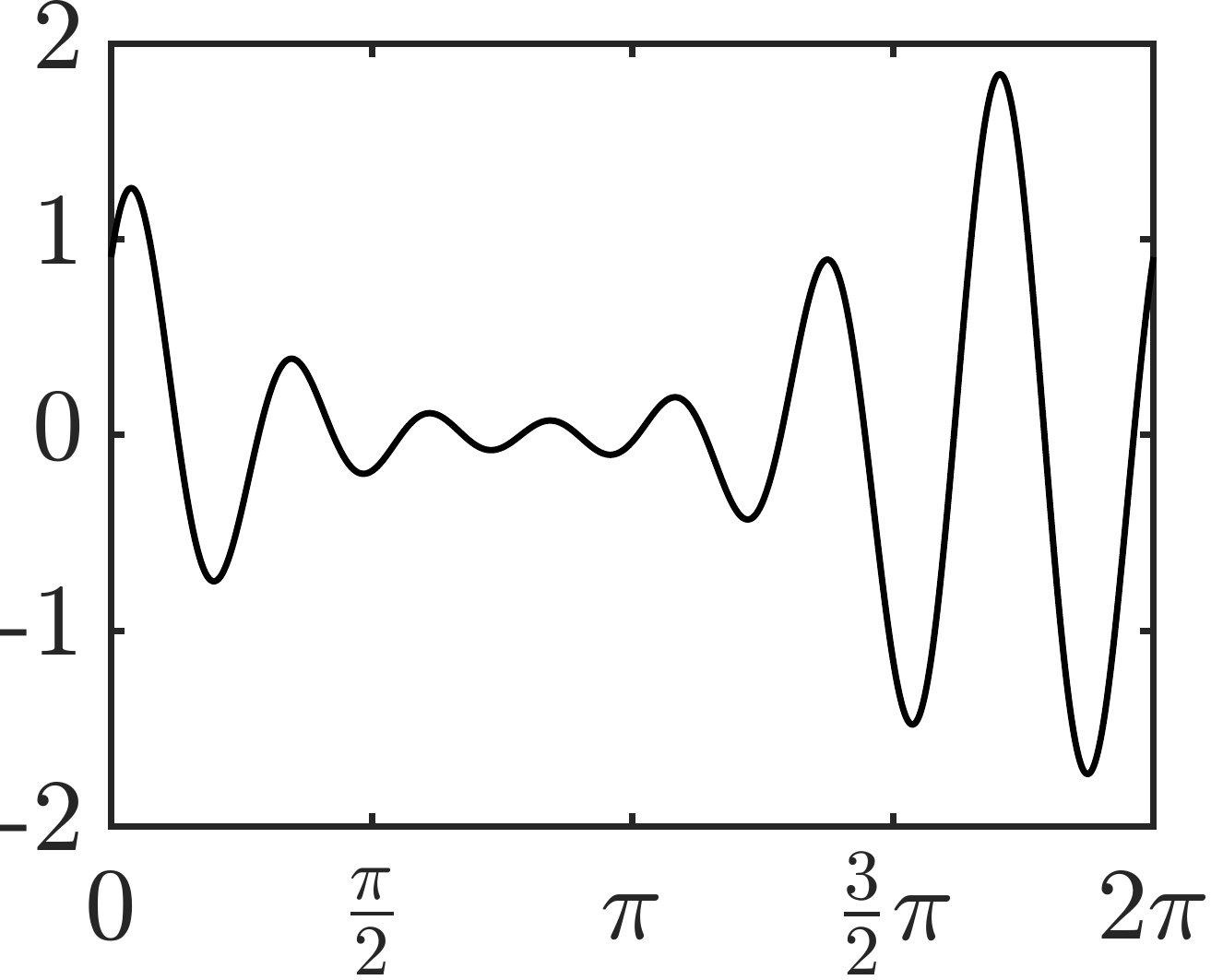}}
\end{tabular}
\end{center}
\caption{Four boundary currents $g_1,\ldots,g_4\in L^2(\partial \Omega)$ for
which the measurements $F(\gamma):=\left(\int_{\partial \Omega} g_j \Lambda(\gamma) g_j \dx[s]\right)_{j=1}^4$
uniquely determine $\gamma\in [a,b]^4\subset \R^4$ by a globally convergent Newton method.}
\label{fig:boundary_currents}
\end{figure}

From checking \eqref{eq:lemma_choose_meas_general_gj}, we also obtain the Lipschitz stability constant
for $F(\gamma):=\left(\int_{\partial \Omega} g_j \Lambda(\gamma) g_j \dx[s]\right)_{j=1}^4$ 
as described in lemma \ref{lemma:F_prop_meas_general}(c). For our setting we obtain
the stability estimate
\begin{equation}\label{eq:numerics_stability}
{} \norm{\gamma-\gamma'}_\infty \leq 7.5 \frac{\norm{F(\gamma)-F(\gamma')}_\infty}{\norm{F(2)-F(1)}_\infty}
\quad \text{ for all } \gamma,\gamma'\in I\supseteq [1,2]^4,
\end{equation}
where $I$ is the slightly enlarged interval from lemma~\ref{lemma:F_prop_meas_general}.

Note that here and in the following we consider the measurement error 
relative to $\norm{F(2)-F(1)}_\infty$ as this is the width of the measuring range
$F(1)\geq F(\gamma)\geq F(2)$.

The property \eqref{eq:lemma_choose_meas_general_gj} can be interpreted in the sense that the boundary current $g_j$
generates an electric potential for which the corresponding solution $|u_{\kappa}^{g_j}|^2$ is much larger on $\Gamma_j$ than on the remaining boundary $\Gamma\setminus \Gamma_j$, and that this simultaneously holds for several (but finitely many) Robin transmission coefficients $\kappa=\kappa^{(j,k)}$. To illustrate this localization property, figure \ref{fig:localization_property} shows $|u_{\kappa}^{g_j}|^2$ (in logarithmic scale) for $k=1$ and $k=K$.

\begin{figure}
\begin{center}
\begin{tabular}{c@{\ } c@{\ } c@{\ } c}
\mbox{\includegraphics[width=0.23\textwidth]{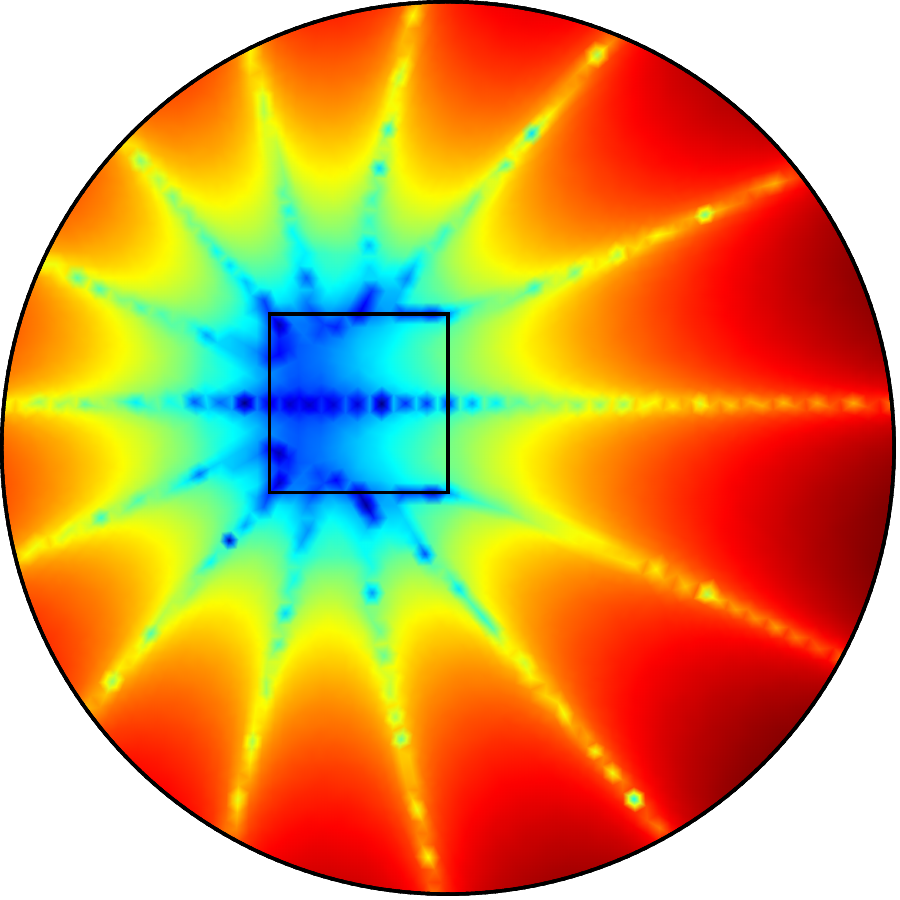}} &
\mbox{\includegraphics[width=0.23\textwidth]{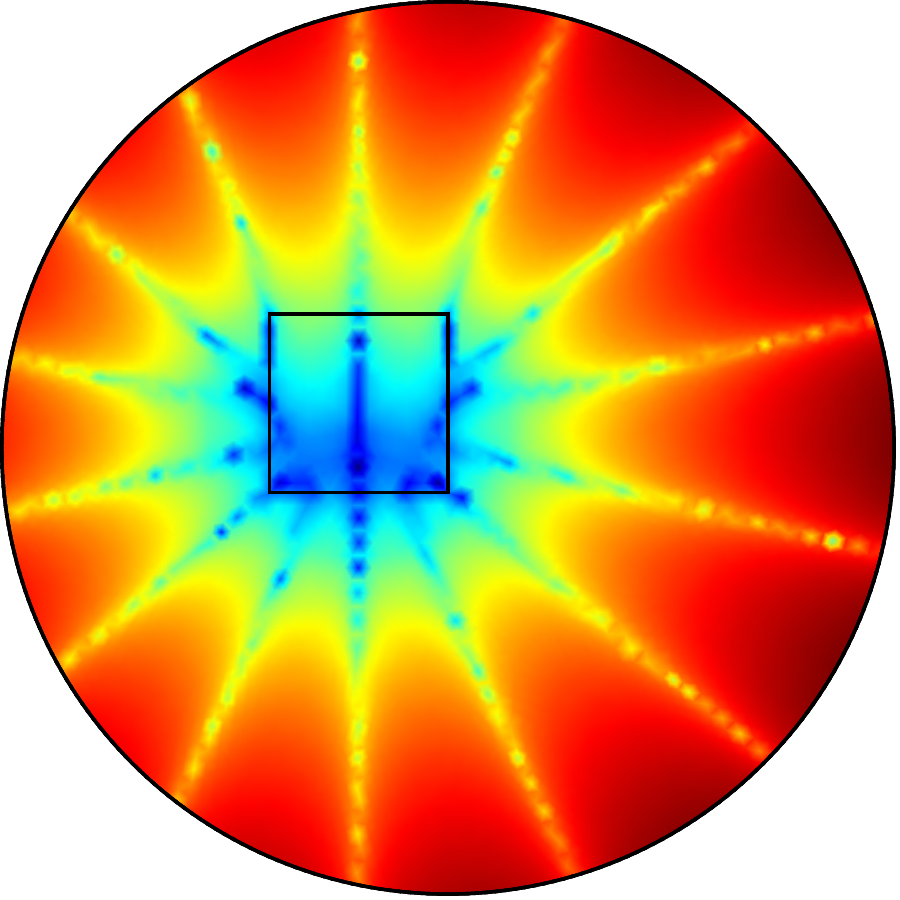}} &
\mbox{\includegraphics[width=0.23\textwidth]{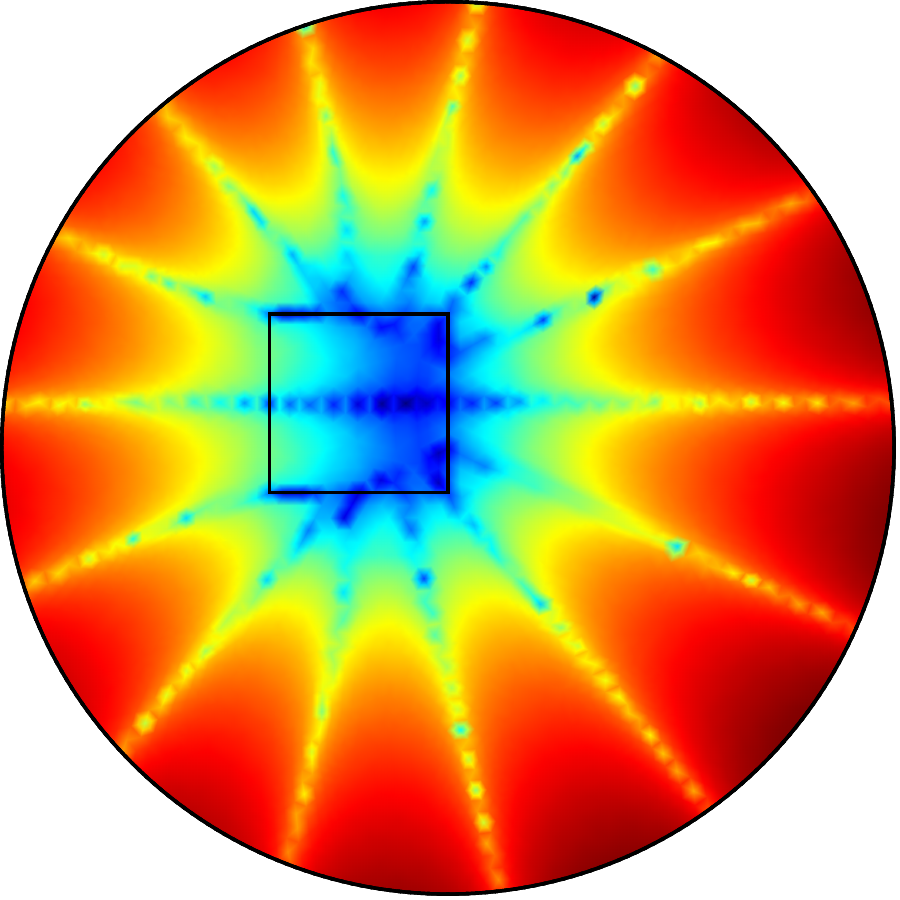}} &
\mbox{\includegraphics[width=0.23\textwidth]{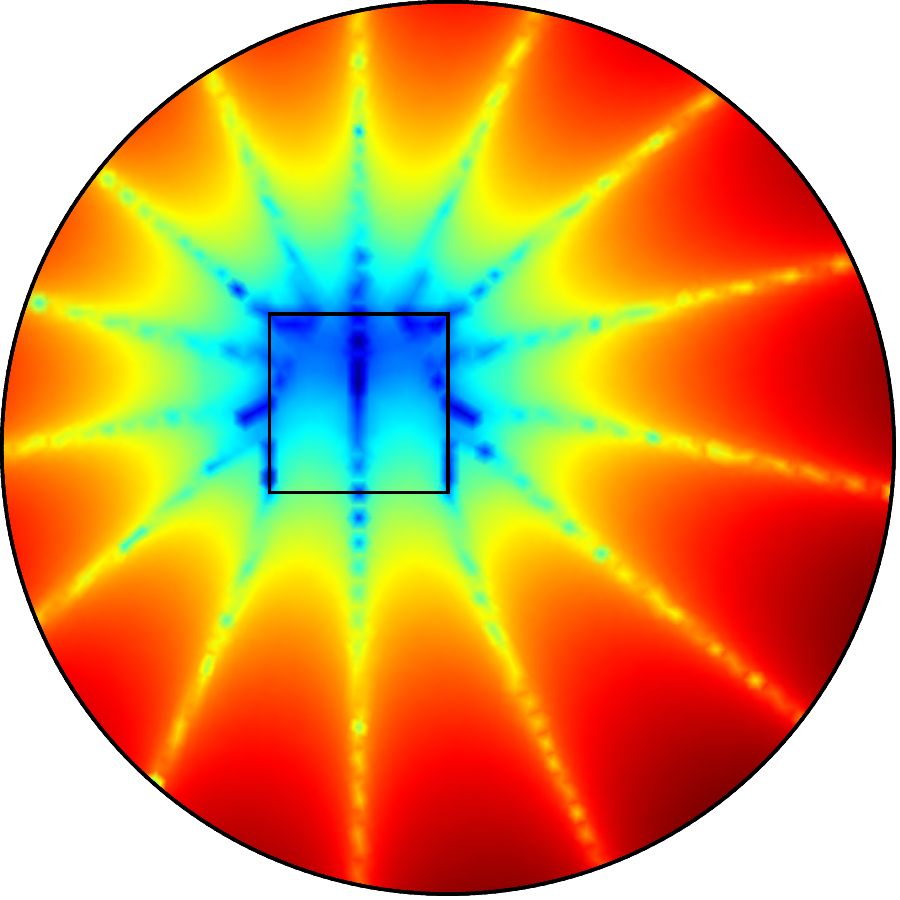}}\\
\mbox{\includegraphics[width=0.23\textwidth]{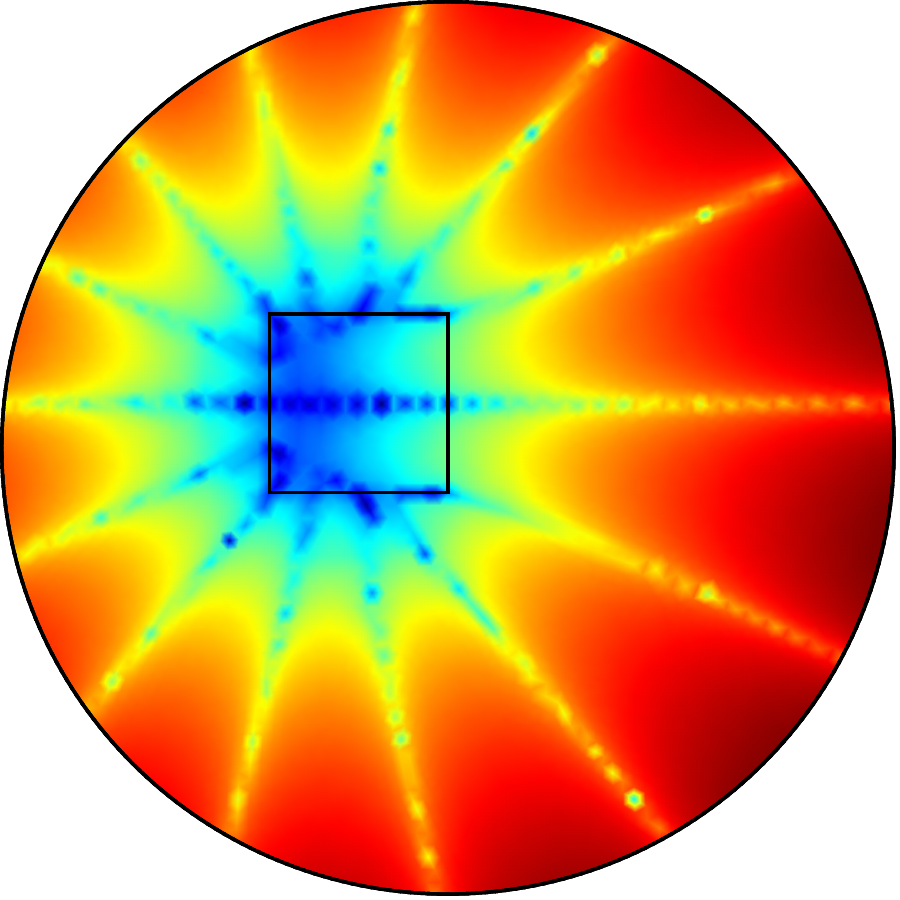}} &
\mbox{\includegraphics[width=0.23\textwidth]{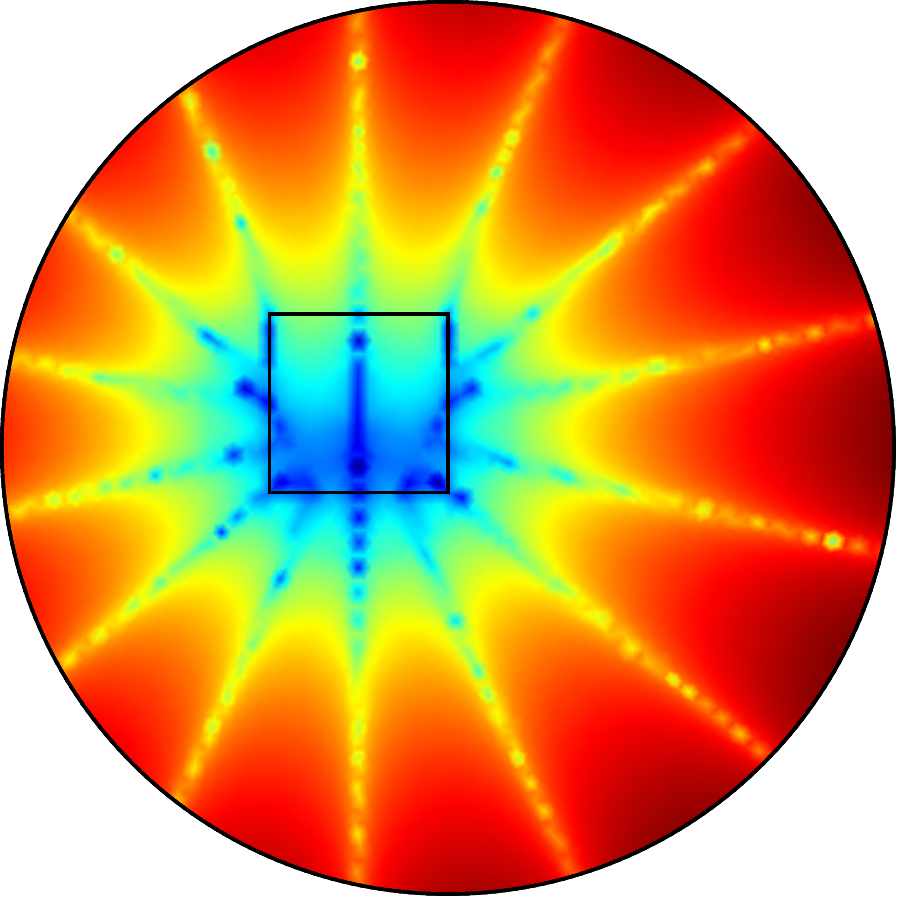}} &
\mbox{\includegraphics[width=0.23\textwidth]{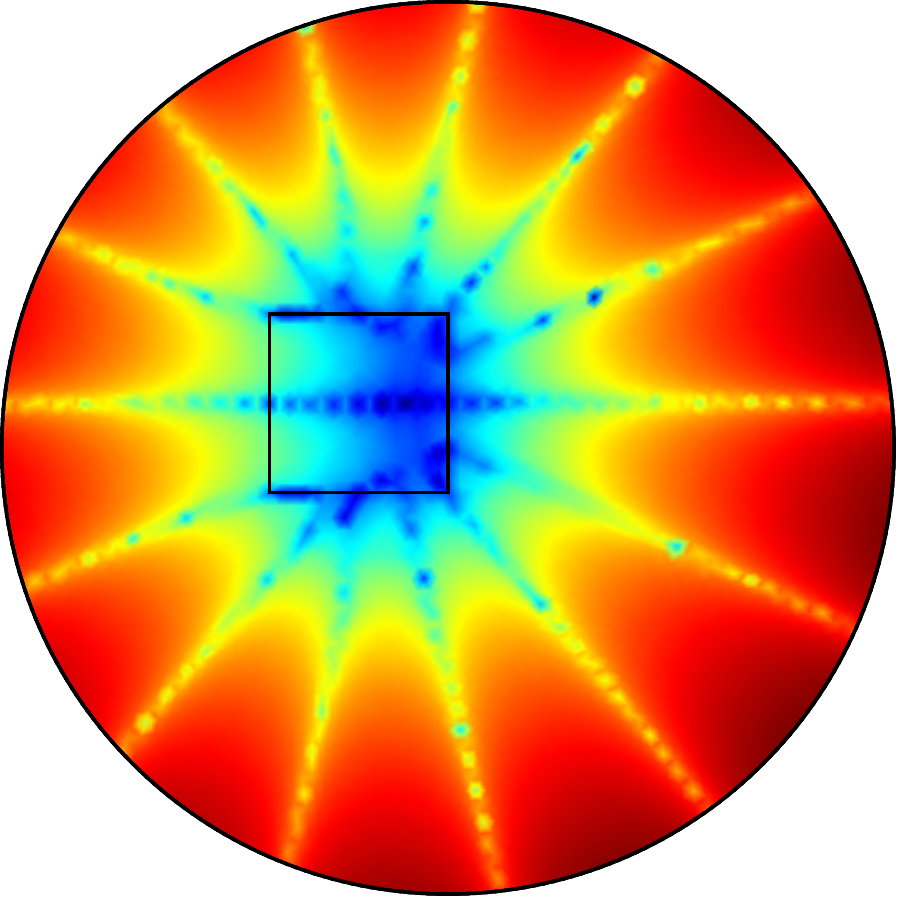}} &
\mbox{\includegraphics[width=0.23\textwidth]{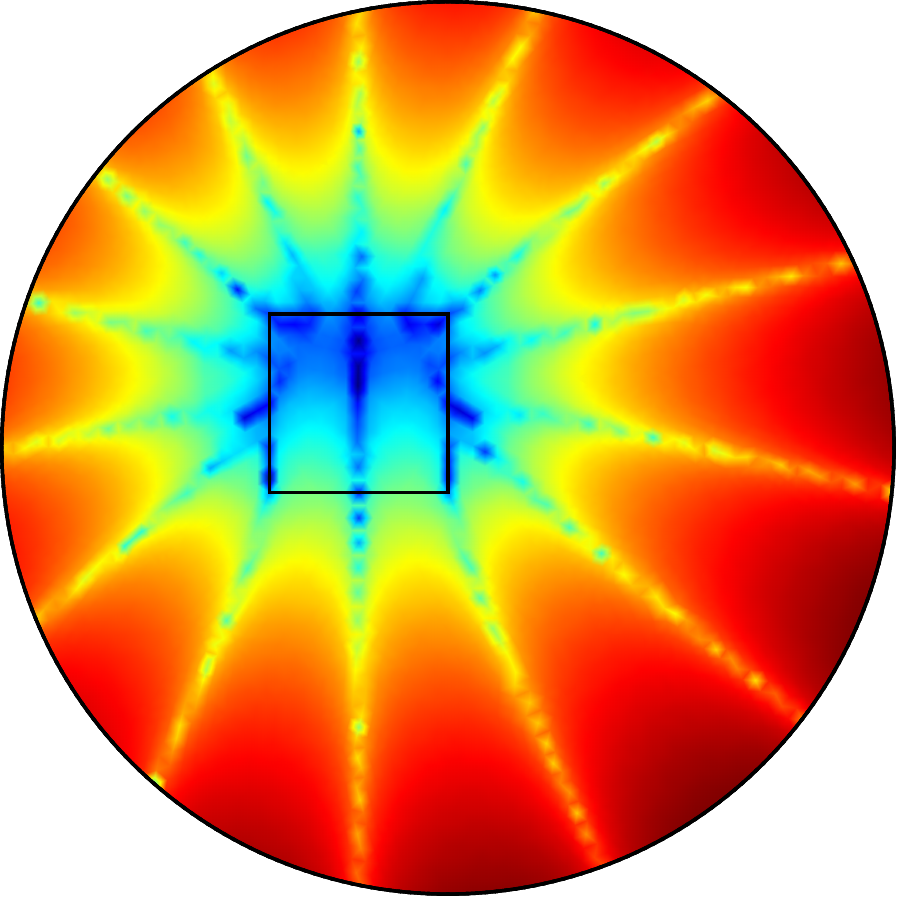}}
\end{tabular}
\end{center}
\caption{Plot of $\log |u_{\kappa}^{g_j}|^2$ generated by the four boundary currents $g_j$ in Figure \ref{fig:boundary_currents} for $j=1,\ldots,4$ (first to fourth column) and $\kappa=\kappa^{(j,k)}$
for $k=1$ (first line) and $k=K$ (last line) showing the localization on boundary part $\Gamma_j$.}
\label{fig:localization_property}
\end{figure}

Let us make a comment on improving the computation time. 
Note that the properties of $F$ only depend on whether the used Neumann functions $g_j$
have the desired property \eqref{eq:lemma_choose_meas_general_gj}, and that our rigorous approach of constructing $g_j$ is computationally more expensive than checking whether some given $g_j$ fulfills \eqref{eq:lemma_choose_meas_general_gj}. In fact, for fixed $j\in \{1,\ldots,n\}$, the construction of the matrix $M_m^{(j)}$ requires solving the PDE \eqref{eq:Robin1}--\eqref{eq:Robin4} for all combinations of $K$ different Robin transmission coefficients and $m$ different Neumann boundary values. On the other hand, checking whether a given Neumann boundary function $g_j$ has the desired property \eqref{eq:lemma_choose_meas_general_gj} only requires solving the PDE \eqref{eq:Robin1}--\eqref{eq:Robin4} for $K$ different Robin transmission coefficients
and the single Neumann boundary function $g_j$. Moreover, as long as $g_j$ does not fulfill  \eqref{eq:lemma_choose_meas_general_gj}, the checking might require very few PDE solutions
if \eqref{eq:lemma_choose_meas_general_gj} already fails to hold for a small $k$.

Hence, one might try computationally cheaper heuristic approaches to construct $g_j$ that satisfy \eqref{eq:lemma_choose_meas_general_gj}. In our experiments, 
we successfully used the ad-hoc approximation
\[
M_m^{(j)}\approx \frac{\alpha}{2} S^{(j)} - \beta \frac{K}{2} (R^{(j,1)}+R^{(j,K)}) \in \R^{m\times m},
\]
(which only requires $2m$ PDE solutions), and always found that increasing the dimension $m\in \N$
lead to Neumann boundary function $g_j$ fulfilling 
\eqref{eq:lemma_choose_meas_general_gj} for all $k\in \{1,\ldots,K\}$.
Moreover, in our experiments, we found the functions $g_j$ constructed with this faster heuristic approach virtually identical
to those constructed with the exact matrix $M_m^{(j)}$ from Theorem \ref{theorem:choose_meas_general}.

\subsubsection{Global convergence of Newton's method}

We numerically study the theoretically predicted global convergence of the standard Newton method
when applied to the measurements $g_1,\ldots,g_4$  constructed in the last subsection. 
We slightly change the definition of $F$ and define the measurements relative to the known lower bound $\gamma=a$
\[
F(\gamma):=\left(\int_{\partial \Omega} g_j \left( \Lambda(\gamma) - \Lambda(a)\right) g_j \dx[s]\right)_{j=1}^4,
\]
and numerically evaluate $F$ using that, for all $g\in L^2(\partial \Omega)$, and $\gamma\in L^\infty_+(\Gamma)$,
\[
\int_{\partial \Omega} g \left( \Lambda(\gamma) - \Lambda(a)\right) g \dx[s]
=\int_\Gamma (a-\gamma) u_a^{(g)} u_\gamma^{(g)}\dx[s],
\]
which immediately follows from the variational formulation of the Robin transmission problem \eqref{Robin_variational}.
Note that this approach is numerically more stable than calculating $\Lambda(\gamma)$ and $\Lambda(a)$ separately as it avoids loss of significance effects.


We choose the true coefficient value as $\hat \gamma:=(1.3,1.8,1.5,1.9)$, and first test the reconstruction for noiseless data $y^\delta:=\hat y:=F(\hat \gamma)$. 
Starting with the lower bound $\gamma^{(0)}:=(1,1,1,1)$, we implement the standard Newton method
\[
\gamma^{(i+1)}:=\gamma^{(i)}-F'(\gamma^{(i)})^{-1} \left( F(\gamma^{(i)}) - y^\delta\right),
\]
where the $(j,l)$-th entry of the Jacobian matrix $F'(\gamma)\in \R^{4\times 4}$ is given
by $\int_{\Gamma_l} |u^{(g_j)}_\gamma|^2 \dx[s]$, cf.\ lemma~\ref{lemma:F_rescaled}.

We repeat this for noisy data with relative noise level $\delta>0$, 
that we obtain by adding a vector with random entries to $\hat y$, 
so that $\norm{y^\delta-\hat y}_\infty=\delta \norm{F(b)}_\infty$.
Note that $F(a)=0$, so that this chooses the norm level relative to the measurement range. For the noiseless case $\delta=0$ we committed the so-called inverse crime of using the same 
forward solver (i.e., the same finite element mesh) for
simulating the data $\hat y=F(\hat \gamma)$ and for evaluating $F$ and $F'$ in the Newton iteration.
For the noisy cases $\delta>0$, we used a different mesh for the forward and inverse solvers. 

\begin{figure}
\begin{center}
\mbox{\includegraphics[height=0.25\textheight]{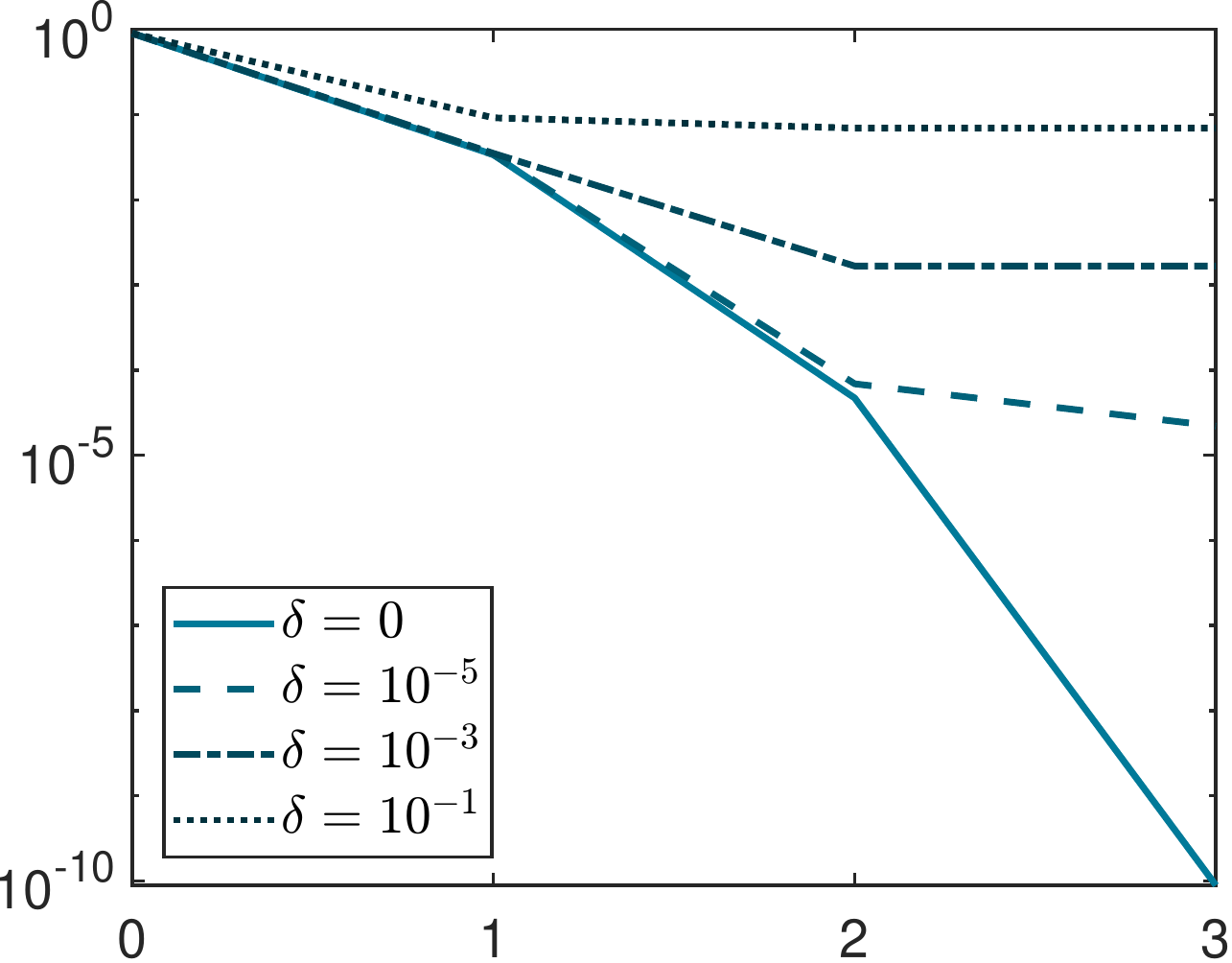}}
\end{center}
\caption{Global convergence of Newton method when applied to the right measurements.}
\label{fig:error_plot}
\end{figure}

Figure \ref{fig:error_plot} shows the error of the first Newton iterations for the case $\delta=0$, $\delta=10^{-5}$, $\delta=10^{-3}$,
and $\delta=10^{-1}$, and demonstrates the theoretically predicted quadratic convergence properties. 
At this point, let us stress that also for noisy data $y^\delta$, Lemma \ref{lemma:F_prop_meas_general} yields that
there exists a unique solution $y^\delta\in I\supseteq [a,b]^n$ of
\[
F(\gamma^\delta)=y^\delta,
\]
and that the standard Newton method converges to this solution $\gamma^\delta$, as long as $y^\delta$ lies within the bounds $F(a)\geq y^\delta \geq F(b)$
(which is easily guaranteed by capping or flooring the values in $y^\delta$).
Moreover, the obtained solution will satisfy the error estimate
\[
\norm{\gamma^\delta-\hat \gamma}_\infty
\leq 7.34\,\delta 
\]
due to the stability estimate \eqref{eq:numerics_stability} obtained in the last subsection.

We finish this subsection with an example where the true Robin transmission coefficient $\hat \gamma\in L^\infty_+(\Gamma)$ is not piecewise constant
but within the a-priori known bounds 
\[
a\leq \hat \gamma(x)\leq b\quad \text{ for $x\in \Gamma$ (a.e.)}
\]
Then $\hat y=\left(\int_{\partial \Omega} g_j \left( \Lambda(\gamma) - \Lambda(a)\right) g_j \dx[s]\right)_{j=1}^4\in \R^4$ will still satisfy
$F(a)\geq \hat y \geq F(b)$, so that there exists a unique solution $\gamma\in I\supseteq [a,b]^n$ 
of $F(\gamma)=\hat y$, i.e., there exists a unique piecewise constant Robin transmission coefficient that leads to the same measured data as the true 
non-piecewise-constant coefficient function. The Newton iteration applied to $\hat y$ (or a noisy version $y^\delta$) will globally converge to this piecewise constant solution
$\gamma$ (or an approximation $\gamma^\delta$), see figure~\ref{fig:nonconstant} for a numerical example.

\begin{figure}
\begin{center}
\mbox{\includegraphics[height=0.25\textheight]{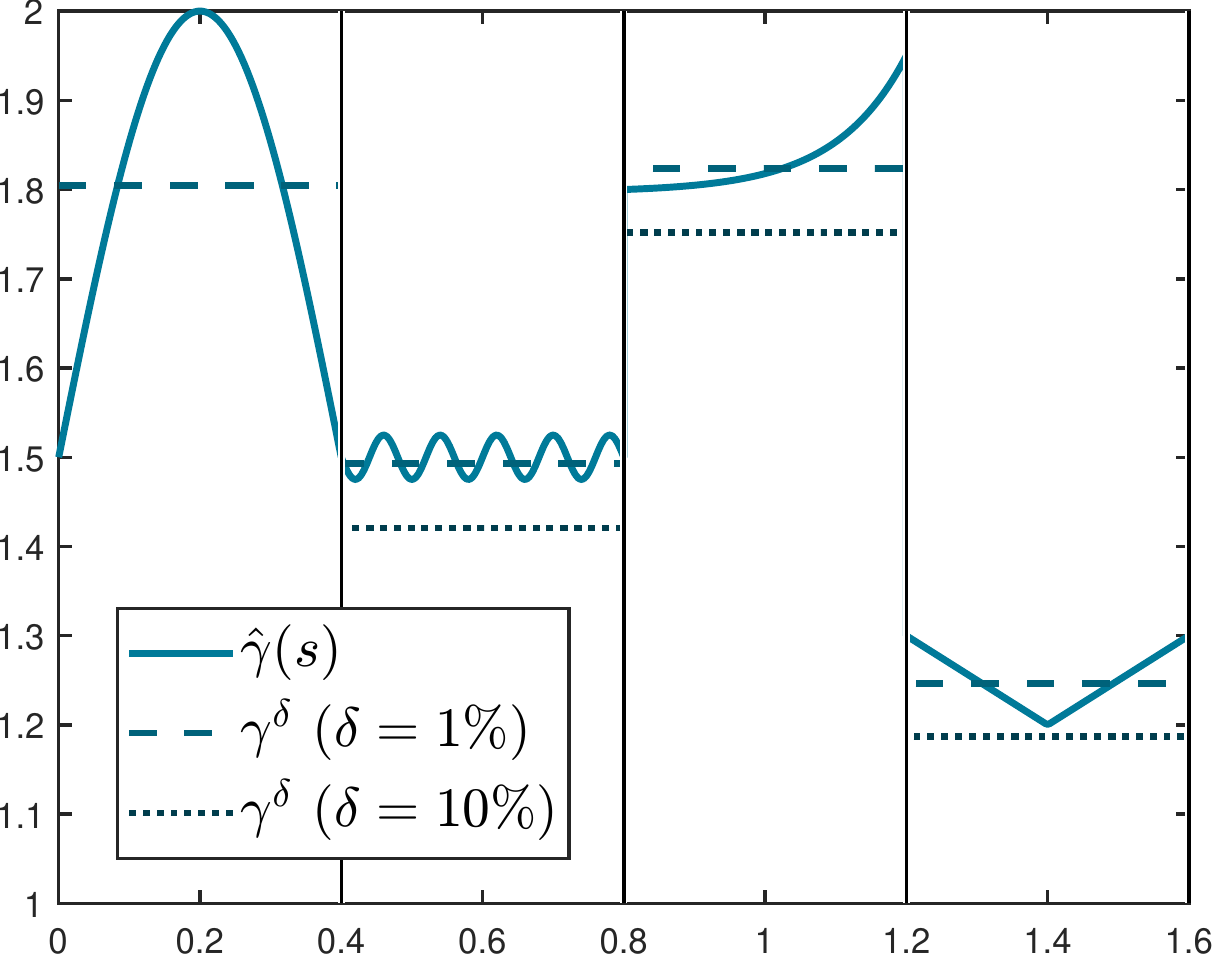}}
\end{center}
\caption{Reconstructions for non-piecewise-constant Robin transmission coefficients.}
\label{fig:nonconstant}
\end{figure}

\subsubsection{Effect of interval width and number or unknowns}

Our result in Theorem \ref{theorem:choose_meas_general} holds for any a-priori known bounds $b>a>0$ and any number of unknowns $n\in \N$. Thus, in theory, we can treat arbitrary large intervals $[a,b]$ and arbitrary fine resolutions of $\Gamma$.
However, numerically, the constructed trigonometric polynomials $g_j$ will quickly become more and more oscillatory, and the calculated Lipschitz constants will quickly increase. 
 
To demonstrate the effect of the interval width, we proceed as in subsection \ref{subsect:measurements} to calculate four boundary currents $g_1,\ldots,g_4\in L^2(\partial \Omega)$ that uniquely determine $\gamma\in [a,b]^4$ and yield global Newton convergence 
for $a=1$, and $b\in \{2,3,4,5\}$. Table \ref{tab:interval_width} shows the dimension of the trigonometric polynomial subspace of $L^2(\partial \Omega)$ that contains $g_1,\ldots,g_4$ and the obtained Lipschitz constant for the inverse problem of determining $\gamma$ from the corresponding measurements.
\begin{table}
\caption{Effect of interval width on trigonometric degree and Lipschitz constant}
\label{tab:interval_width} 
\begin{tabular}{c c c}
\hline\noalign{\smallskip}
interval & dimension $m$ & Lipschitz constant\\
\noalign{\smallskip}\hline\noalign{\smallskip}
$[1,2]$ & $15$ & $7.34$\\
$[1,3]$ & $21$ & $21.0$\\
$[1,4]$ & $21$ & $35.8$ \\
$[1,5]$ & $71$ & $458000$ \\
\noalign{\smallskip}\hline
\end{tabular}
\end{table}

To demonstrate the effect of the number of unknowns, we then replace the square $D$ by regular polygons with $n=3$, $n=4$, $n=5$, and $n=6$ sides keeping the polygon center and circumradius the same as in the square ($n=4$) case. We assume that $\gamma$ is piecewise constant with respect to the polygon sides. As in subsection \ref{subsect:measurements} we then calculate $n$ boundary currents $g_1,\ldots,g_n\in L^2(\partial \Omega)$ that uniquely determine $\gamma\in [1,2]^n$ and yield global Newton convergence. The required dimension of the trigonometric polynomial subspace of $L^2(\partial \Omega)$ and the obtained Lipschitz constant are shown in table \ref{tab:number_unknowns}.
\begin{table}
\caption{Effect of number of unknowns $n$ on trigonometric degree and Lipschitz constant}
\label{tab:number_unknowns} 
\begin{tabular}{c c c}
\hline\noalign{\smallskip}
interval & dimension $m$ & Lipschitz constant\\
\noalign{\smallskip}\hline\noalign{\smallskip}
$n=3$ & $11$ & $6.37$\\
$n=4$ & $15$ & $7.34$\\
$n=5$ & $23$ & $54.2$ \\
$n=6$ & $55$ & $507$ \\
\noalign{\smallskip}\hline
\end{tabular}
\end{table}

In both situations, the boundary currents quickly become highly oscillatory, and the calculated stability constant worsens. Hence, at the current state, our approach will only be feasible for moderate contrasts and relatively few unknowns as stated in the introduction. It should be noted, however, that our criterion \eqref{eq:lemma_choose_meas_general_gj} in Theorem \ref{theorem:choose_meas_general} is sufficient but possibly not necessary for uniqueness, Lipschitz stability and global Newton convergence. The constructed boundary currents and the calculated Lipschitz constants may be far from optimal. Since our result is (to the knowledge of the author) the very first on uniqueness, global convergence and explicit Lipschitz stability constants for a discretized inverse coefficient problem, there may well be room for improvement and significantly sharper estimates that could practically yield in less oscillations and better stability constants.

\section{Conclusions}
We have derived a method to determine which (finitely many) measurements uniquely determine the unknown coefficient in an inverse coefficient problem with a given resolution,
and proved global convergence of Newton's method for the resulting discretized non-linear inverse problem. Our method also allows to explicitly calculate the Lipschitz stability constant, and yields an error estimate for noisy data. To the knowledge of the author, these are the first such results for discretized inverse coefficient problems. 

Our method stems on an extension of classical global Newton convergence theory from convex inverse-monotonic to convex (forward-)monotonic functions that
arise in elliptic inverse coefficient problems. The extension required an extra assumption on the directional derivatives of the considered function
that we were able to fulfill by choosing the right measurements. 

Our proofs mainly utilized monotonicity ideas and localized potential techniques that are also known for several other elliptic inverse coefficient problems. So the ideas in this work might be applicable to other applications as well. A particularly interesting extension would be the case of EIT where it has recently been shown \cite{harrach2019uniqueness} that an unknown conductivity distribution with a given resolution is uniquely determined by voltage-current-measurements on sufficiently many electrodes, but the number of required electrodes is not known. The main difficulty of such an extension is that localized potentials in EIT cannot concentrate on each domain part separately as in the simpler Robin transmission problem considered in this work. Roughly speaking, a localized potential in EIT with high energy in some region $D$ will also have a high energy on its way from the boundary electrodes to $D$. 
This behavior will make the application of our herein presented ideas more challenging.

\section*{Acknowledgments}
The author would like to thank Professor Michael Klibanov for his inspiring work on global convergence, and Professor Frank Natterer for
pointing out possible relations between the monotonicity method and Collatz monotone functions.

\bibliography{literaturliste.bib}

\bibliographystyle{spmpsci}

\end{document}